\documentclass[english]{article}
\usepackage[T1]{fontenc}
\usepackage[latin9]{inputenc}
\usepackage{geometry}
\usepackage{color}
\usepackage{babel}
\usepackage{mathrsfs}
\usepackage{amsmath}
\usepackage{amsthm}
\usepackage{amssymb}
\usepackage[all]{xy}
\usepackage[unicode=true,pdfusetitle,
 bookmarks=true,bookmarksnumbered=false,bookmarksopen=false,
 breaklinks=false,pdfborder={0 0 0},pdfborderstyle={},backref=false,colorlinks=true]
 {hyperref}

\makeatletter
\newcommand\thmsname{\protect\theoremname}
\newcommand\nm@thmtype{theorem}
\theoremstyle{plain}

\newenvironment{namedthm}[1][Undefined Theorem Name]{
  \ifx{#1}{Undefined Theorem Name}\renewcommand\nm@thmtype{theorem*}
  \else\renewcommand\thmsname{#1}\renewcommand\nm@thmtype{namedtheorem}
  \fi
  \begin{\nm@thmtype}}
  {\end{\nm@thmtype}}
\theoremstyle{remark}
\newtheorem*{acknowledgement*}{\protect\acknowledgementname}
\theoremstyle{plain}
\newtheorem{thm}{\protect\theoremname}[section]
\theoremstyle{definition}
\newtheorem{defn}[thm]{\protect\definitionname}
\theoremstyle{plain}
\newtheorem{prop}[thm]{\protect\propositionname}
\theoremstyle{remark}
\newtheorem{rem}[thm]{\protect\remarkname}
\theoremstyle{plain}
\newtheorem{lem}[thm]{\protect\lemmaname}
\theoremstyle{plain}
\newtheorem{cor}[thm]{\protect\corollaryname}
\theoremstyle{remark}
\newtheorem{notation}[thm]{\protect\notationname}
\theoremstyle{definition}
\newtheorem{example}[thm]{\protect\examplename}
\newcommand{\lyxaddress}[1]{
	\par {\raggedright #1
	\vspace{1.4em}
	\noindent\par}
}

\usepackage{amsmath}

\usepackage{mathpazo}

\makeatother

\providecommand{\acknowledgementname}{Acknowledgement}
\providecommand{\corollaryname}{Corollary}
\providecommand{\definitionname}{Definition}
\providecommand{\examplename}{Example}
\providecommand{\lemmaname}{Lemma}
\providecommand{\notationname}{Notation}
\providecommand{\propositionname}{Proposition}
\providecommand{\remarkname}{Remark}
\providecommand{\theoremname}{Theorem}

\begin{document}
\title{\textbf{\textsc{Some computations of stable twisted homology for mapping
class groups}}}
\author{\textsc{\textcolor{black}{Arthur Soulié}}}
\maketitle
\begin{abstract}
In this paper, we deal with stable homology computations with twisted
coefficients for mapping class groups of surfaces and of $3$-manifolds,
automorphism groups of free groups with boundaries and automorphism
groups of certain right-angled Artin groups. On the one hand, the
computations are led using semidirect product structures arising naturally
from these groups. On the other hand, we compute the stable homology
with twisted coefficients by $FI$-modules. This notably uses a decomposition
result of the stable homology with twisted coefficients for pre-braided
monoidal categories proved in this paper.
\end{abstract}
{\let\thefootnote\relax\footnotetext{This work was partially supported by the ANR Project {\em ChroK}, {\tt ANR-16-CE40-0003} and by a JSPS Postdoctoral Fellowship Short Term.\\2010 \textit{Mathematics Subject Classification}: 
 20J05, 20J06, 20F38, 18D10.\\Keywords: Stable homology, Mapping class groups, Functor categories, Polynomial functors.}}

\section*{Introduction}

Computing the homology of a group is a fundamental question and can
be a very difficult task. For example, a complete understanding of
all the homology groups of mapping class groups of surfaces and $3$-manifolds
remains out of reach at present time: this is an active research topic
(see \cite{Harer1985,HatcherWahl2,MadsenWeiss} for constant coefficients
and \cite{harer1991third,Kawa1} for twisted coefficients).

In \cite{WahlRandal-Williams}, Randal-Williams and Wahl prove homological
stability with twisted coefficients for some families of groups, including
mapping class groups of surfaces and $3$-manifolds. They consider
a set of groups $\left\{ G_{n}\right\} _{n\in\mathbb{N}}$ such that
there exist a canonical injections $G_{n}\hookrightarrow G_{n+1}$.
Let $\mathcal{G}$ be the groupoid with objects indexed by natural
numbers, with the groups $\left\{ G_{n}\right\} _{n\in\mathbb{N}}$
as automorphism groups and with no morphisms between distinct objects.
We consider Quillen's bracket construction on $\mathcal{G}$ (see
\cite[p.219]{graysonQuillen}), denoted by $\mathfrak{U}\mathcal{G}$,
and $\mathbf{Ab}$ the category of abelian groups. Randal-Williams
and Wahl show that for particular kinds of functors $F:\mathfrak{U}\mathcal{G}\rightarrow\mathbf{Ab}$
(namely coefficients systems of finite degree, see \cite[Section 4]{WahlRandal-Williams}),
then the canonical induced maps
\[
H_{*}\left(G_{n},F\left(n\right)\right)\rightarrow H_{*}\left(G_{n+1},F\left(n+1\right)\right)
\]
are isomorphisms for $N\left(*,d\right)\leq n$ with some $N\left(*,d\right)\in\mathbb{N}$
depending on $*$ and $d$. The value of the homology for $n\geq N\left(*,d\right)$
is called the stable homology of the family of groups $\left\{ G_{n}\right\} _{n\in\mathbb{N}}$
and denoted by $H_{*}\left(G_{\infty},F_{\infty}\right)$.

In this paper, we are interested in explicit computations of the stable
homology with twisted coefficients for mapping class groups. On the
one hand, using semidirect product structures naturally arising from
mapping class groups, on the strength of Lyndon-Hochschild-Serre spectral
sequence and of certain stability results, we prove:
\begin{namedthm}[\textit{Theorem A }(Theorems \ref{thm:homologymcgtensorproduct} and
\ref{thm:resultAutFnwithbound})]
We have:

\begin{enumerate}
\item We denote by $\Gamma_{g,1}$ the isotopy classes of diffeomorphisms
restricting to the identity on the boundary component of a compact
connected orientable surface with one boundary component and genus
$g\geq0$. Then, for $m$, $n$ and $q$ natural numbers such that
$2n\geq3q+m$, there is an isomorphism:
\[
H_{q}\left(\Gamma_{n,1},H_{1}\left(\varSigma_{n,1},\mathbb{Z}\right)^{\otimes m}\right)\cong\underset{\left\lfloor \frac{q-1}{2}\right\rfloor \geq k\geq0}{\bigoplus}H_{q-(2k+1)}\left(\Gamma_{n,1},H_{1}\left(\varSigma_{n,1},\mathbb{Z}\right)^{\otimes m-1}\right).
\]
Hence, we recover the results of \cite{harer1991third} and \cite{Kawa1}.
\item We consider the graph $\mathscr{G}_{n,k}^{s}$ defined as the wedge
of $n\in\mathbb{N}$ circles with $k\in\mathbb{N}$ distinguished
circles joined by arcs to the basepoint $p_{0}$ and $s-1\in\mathbb{N}$
extra basepoints joined by new edges to $p_{0}$. We denote by $A_{n,k}^{s}$
the group of path-components space of homotopy equivalences of the
graph $\mathscr{G}_{n,k}^{s}$ (we refer the reader to Section \ref{subsec:autFnwithboundaries}
for a complete introduction to these groups). In particular it is
the quotient of the mapping class group of some $3$-manifold (see
Remark \ref{rem:autfreewithboundquotient}). For $s\geq2$ and $q\geq1$
be natural numbers and $F:\mathfrak{gr}\rightarrow\mathbf{Ab}$ a
reduced polynomial functor where $\mathfrak{gr}$ denotes the category
of finitely generated free groups. There is an isomorphism:
\[
H_{q}\left(A_{\infty,0}^{s},F_{\infty}\right)=0.
\]
Moreover, $H_{q}\left(A_{n,k}^{s},\mathbb{Q}\right)=0$ for all  natural
numbers $n\geq3q+3$ and $k\geq0$. We thus recover the results of
\cite{jensenhol} for holomorphs of free groups.
\end{enumerate}
\end{namedthm}
On the other hand, we deal with stable homology computations for mapping
class groups with twisted coefficients factoring through some finite
groups. Let $\left(\Sigma,\sqcup,\emptyset\right)$ be the symmetric
monoidal groupoid with objects the finite sets, with automorphism
groups the symmetric groups and with no morphisms between distinct
objects, the monoidal structure is given by the disjoint union $\sqcup$.
Quillen's bracket construction $\mathfrak{U}\Sigma$ is equivalent
to the category $FI$ of finite sets and injections used in \cite{ChurchEllenbergFarbstabilityFI}.
For $R$ a commutative ring, $R\textrm{-}\mathfrak{Mod}$ denotes
the category of $R$-modules. We prove the following results.
\begin{namedthm}[\textit{Theorem B }(Proposition \ref{prop:Bigresultbraidgroup}, Proposition
\ref{prop:Bigresultmcgorient}, Example \ref{exa:caseofautfree})]
Let $\mathbb{K}$ be a field of characteristic zero and $d$ be a
natural number. Considering functors $F:FI\rightarrow\mathbb{K}\textrm{-}\mathfrak{Mod}$,
we have:

\begin{enumerate}
\item For $n$ a natural number, we respectively denote by $\mathbf{B}_{n}$
the braid group on $n$ strands and by $\mathbf{PB}_{n}$ the pure
braid group on $n$ strands. Then, $H_{d}\left(\mathbf{B}_{\infty},F_{\infty}\right)\cong\underset{n\in FI}{\textrm{Colim}}\left(H_{d}\left(\mathbf{PB}_{n},\mathbb{K}\right)\underset{\mathbb{K}}{\otimes}F\left(n\right)\right)$.
\item $H_{d}\left(\Gamma_{\infty,1}^{\infty},F_{\infty}\right)\cong\underset{n\in FI}{\textrm{Colim}}\left[\underset{k+l=d}{\bigoplus}\left(H_{k}\left(\Gamma_{n,1},\mathbb{K}\right)\underset{\mathbb{K}}{\otimes}H_{l}\left(\left(\mathbb{C}P^{\infty}\right)^{\times n},\mathbb{K}\right)\right)\underset{\mathbb{K}}{\otimes}F\left(n\right)\right]$,
where $\Gamma_{g,1}^{s}$ denotes the isotopy classes of diffeomorphisms
permuting the marked points and restricting to the identity on the
boundary component of a compact connected orientable surface with
one boundary component, genus $g\geq0$ and $s\geq0$ marked points.
It follows from Madsen-Weiss theorem \cite{MadsenWeiss} that $H_{2k+1}\left(\Gamma_{\infty,1}^{\infty},F_{\infty}\right)=0$
for all natural numbers $k$.
\item $H_{d}\left(Aut\left(\left(\mathbb{Z}^{*k}\right)^{\times\infty}\right),F_{\infty}\right)=0$
for a fixed natural number $k\geq2d+1$ (where $*$ denotes the free
product of groups).
\end{enumerate}
\end{namedthm}
Moreover, we may make further explicit calculations for specific polynomial
$FI$-modules $F$ using the formulas of Theorem $B$, although it
generally requires some non-trivial extra work: the key point is to
understand the $FI$-module structure of the left-hand terms in the
formula. Namely, the action of the symmetric group $\mathfrak{S}_{n}$
on the homology group $H_{q}\left(\mathbf{PB}_{n},\mathbb{K}\right)$
is studied in \cite[Example 5.1.A]{ChurchEllenbergFarbstabilityFI}
using the computations of \cite{arnold} and the one on $\underset{k+l=d}{\bigoplus}\left(H_{k}\left(\Gamma_{-,1},\mathbb{K}\right)\underset{\mathbb{K}}{\otimes}H_{l}\left(\left(\mathbb{C}P^{\infty}\right)^{\times-},\mathbb{K}\right)\right)$
is induced by the action of the symmetric groups on the homology of
$\mathbb{C}P^{\infty}$. Therefore we can understand the $FI$-module
structure of the associated pointwise tensor product and then compute
the colimit with respect to $FI$.

The proof of Theorem $B$ requires a splitting result for the twisted
stable homology for some families of  groups: this decomposition consists
in the graded direct sum of tensor products of the homology of an
associated category with the stable homology with constant coefficients.
Namely, we assume that the category $\left(\mathfrak{U}\mathcal{G},\natural,0\right)$
is pre-braided  homogeneous (we refer the reader to Section \ref{sec:Appendix}
for an introduction to these notions) such that the unit $0$ is an
initial object. For a functor $F:\mathfrak{U}\mathcal{G}\rightarrow\mathbf{Ab}$,
we denote by $H_{*}\left(\mathfrak{U}\mathcal{G},F\right)$ the homology
of the category $\mathfrak{U}\mathcal{G}$ with coefficient in $F$
(we refer the reader to the papers \cite[Section 2]{FranjouPira}
and \cite[Appendice A]{DV1} for an introduction to this last notion).
We prove the following statement.
\begin{namedthm}[\textit{Theorem C }(Theorem \ref{thm:step1}) ]
Let $\mathbb{K}$ be a field. For all functors $F:\mathfrak{U}\mathcal{G}\rightarrow\mathbb{K}\textrm{-}\mathfrak{Mod}$,
we have a natural isomorphism of $\mathbb{K}$-modules:
\[
H_{*}\left(G_{\infty},F_{\infty}\right)\cong\underset{k+l=*}{\bigoplus}\left(H_{k}\left(G_{\infty},\mathbb{K}\right)\underset{\mathbb{K}}{\otimes}H_{l}\left(\mathfrak{U}\mathcal{G},F\right)\right).
\]
\end{namedthm}
If the groupoid $\mathcal{G}$ is symmetric monoidal, then Theorem
$C$ recovers the previous analogous results \cite[Propositions 2.22, 2.26]{DV1}.

The paper is organized as follows. In Section \ref{sec:Appendix},
we recall notions on Quillen's bracket construction, pre-braided monoidal
categories and  homogeneous categories. In Section \ref{sec:sssp},
after setting up the general framework for the families of groups
we will deal with and applying Lyndon-Hochschild-Serre spectral sequence,
we prove the various results of Theorem $A$. In Section \ref{sec:FI-Mod},
the first part is devoted to the proof of the decomposition result
Theorem $C$. Then we deal with the twisted stable homology for mapping
class groups with non-trivial finite quotient groups and prove Theorem
$B$.

\paragraph{General notations.}

We fix $R$ a commutative ring and $\mathbb{K}$ a field throughout
this work. We denote by $R\textrm{-}\mathfrak{Mod}$ the categories
of $R$-modules.

Let $\mathfrak{Cat}$ denote the category of small categories. Let
$\mathfrak{C}$ be an object of $\mathfrak{Cat}$. We use the abbreviation
$Obj\left(\mathfrak{C}\right)$ to denote the objects of $\mathfrak{C}$.
For $\mathfrak{D}$ a category, we denote by $\mathbf{Fct}\left(\mathfrak{C},\mathfrak{D}\right)$
the category of functors from $\mathfrak{C}$ to $\mathfrak{D}$.
If $0$ is initial object in the category $\mathfrak{C}$, then we
denote by $\iota_{A}:0\rightarrow A$ the unique morphism from $0$
to $A$. The maximal subgroupoid $\mathscr{G}\mathfrak{r}\left(\mathfrak{C}\right)$
is the subcategory of $\mathfrak{C}$ which has the same objects as
$\mathfrak{C}$ and of which the morphisms are the isomorphisms of
$\mathfrak{C}$. We denote by $\mathscr{G}\mathfrak{r}:\mathfrak{Cat}\rightarrow\mathfrak{Cat}$
the functor which associates to a category its maximal subgroupoid.

We take the convention that the set of natural numbers $\mathbb{N}$
is the set of nonnegative integers $\left\{ 0,1,2,\ldots\right\} $.
We denote by $\left(\mathbb{N},\leq\right)$ the category of natural
numbers considered as a directed set. For all natural numbers $n$,
we denote by $\gamma_{n}$ the unique element of $Hom_{\left(\mathbb{N},\leq\right)}\left(n,n+1\right)$.
For all natural numbers $n'$ such that $n'\geq n$, we denote by
$\gamma{}_{n,n'}:n\rightarrow n'$ the unique element of $Hom_{\left(\mathbb{N},\leq\right)}\left(n,n'\right)$,
composition of the morphisms $\gamma{}_{n'-1}\circ\gamma{}_{n'-2}\circ\cdots\circ\gamma{}_{n+1}\circ\gamma{}_{n}$.
The addition defines a strict monoidal structure on $\left(\mathbb{N},\leq\right)$,
denoted by $\left(\left(\mathbb{N},\leq\right),+,0\right)$.

We denote by $\mathfrak{Gr}$ the category of groups, by $*$ the
coproduct in this category, by $\mathbf{Ab}$ the full subcategory
of $\mathfrak{Gr}$ of abelian groups and by $\mathfrak{gr}$ the
full subcategory of $\mathfrak{Gr}$ of finitely generated free groups.
Recall that the free product of groups $*$ defines a monoidal structure
over $\mathfrak{gr}$, with the trivial group $0_{\mathfrak{Gr}}$
the unit, denoted by $\left(\mathfrak{gr},*,0_{\mathfrak{Gr}}\right)$.
We denote by $\times$ the direct product of groups and by $Aut\left(G\right)$
the automorphism group of a group $G$.
\begin{acknowledgement*}
The author wishes to thank most sincerely Christine Vespa, Nathalie
Wahl and Geoffrey Powell, for their reading, suggestions and advice.
He would also like to thank Aurélien Djament, Nariya Kawazumi and
Antoine Touzé for the attention they have paid to his work and helpful
discussions. Additionally, he would like to thank the anonymous referee
for his careful reading, comments and corrections.
\end{acknowledgement*}
\tableofcontents{}

\section{Categorical framework\label{sec:Appendix}}

This section recollects the notions of Quillen's bracket construction,
pre-braided monoidal categories and  homogeneous categories for the
convenience of the reader. It takes up the framework of \cite[Section 1]{WahlRandal-Williams}.
For an introduction to braided monoidal categories, we refer to \cite[Section XI]{MacLane1}.
Standardly, a strict monoidal category will be denoted by $\left(\mathfrak{C},\natural,0\right)$,
where $\natural:\mathfrak{C}\times\mathfrak{C}\rightarrow\mathfrak{C}$
is the monoidal structure and $0$ is the monoidal unit. If the category
is braided, we denote by $b_{-,-}^{\mathfrak{C}}$ its braiding. \textbf{We
fix a strict monoidal groupoid $\left(\mathfrak{G},\natural,0\right)$
throughout this section.}

\paragraph{Quillen's bracket construction.}

The following definition is a particular case of a more general construction
of \cite{graysonQuillen}.
\begin{defn}
\cite[Section 1.1]{WahlRandal-Williams}\label{def:defUG} Quillen's
bracket construction on the groupoid $\mathfrak{G}$, denoted by $\mathfrak{UG}$
is the category defined by:

\begin{itemize}
\item Objects: $Obj\left(\mathfrak{UG}\right)=Obj\left(\mathfrak{G}\right)$;
\item Morphisms: for $A$ and $B$  objects of $\mathfrak{G}$, $Hom_{\mathfrak{UG}}\left(A,B\right)=\underset{\mathfrak{G}}{colim}\left[Hom_{\mathfrak{G}}\left(-\natural A,B\right)\right]$.
A morphism from $A$ to $B$ in the category $\mathfrak{UG}$ is an
equivalence class of pairs $\left(X,f\right)$, where $X$ is an object
of $\mathfrak{G}$ and $f:X\natural A\rightarrow B$ is a morphism
of $\mathfrak{G}$; this is denoted by $\left[X,f\right]:A\rightarrow B$.
\item Let $\left[X,f\right]:A\rightarrow B$ and $\left[Y,g\right]:B\rightarrow C$
be  morphisms in the category $\mathfrak{UG}$. Then, the composition
in the category $\mathfrak{UG}$ is defined by $\left[Y,g\right]\circ\left[X,f\right]=\left[Y\natural X,g\circ\left(id_{Y}\natural f\right)\right]$.
\end{itemize}
\end{defn}

It is clear that the unit $0$ of the monoidal structure of the groupoid
$\left(\mathfrak{G},\natural,0\right)$ is an initial object in the
category $\mathfrak{U\mathfrak{G}}$ (see \cite[Proposition 1.8 (i)]{WahlRandal-Williams}). 
\begin{defn}
\label{def:nozerodivisors-}The strict monoidal category $\left(\mathfrak{G},\natural,0\right)$
is said to have no zero divisors if for all objects $A$ and $B$
of $\mathfrak{G}$, $A\natural B\cong0$ if and only if $A\cong B\cong0$.
\end{defn}

\begin{prop}
\cite[Proposition 1.7]{WahlRandal-Williams} Assume that the strict
monoidal groupoid $\left(\mathfrak{G},\natural,0\right)$ has no zero
divisors and that $Aut_{\mathfrak{G}}(0)=\left\{ id_{0}\right\} $.
Then, the groupoid $\mathfrak{G}$ is the maximal subgroupoid of $\mathfrak{U\mathfrak{G}}$.
\end{prop}

\textbf{Henceforth, we assume that the groupoid $\left(\mathfrak{G},\natural,0\right)$
has no zero divisors and that $Aut_{\mathfrak{G}}(0)=\left\{ id_{0}\right\} $.}
\begin{rem}
\label{rem:canonicalfunctor}Let $X$ be an object of $\mathfrak{G}$.
Let $\phi\in Aut_{\mathfrak{G}}\left(X\right)$. Then, as an element
of $Hom_{\mathfrak{UG}}\left(X,X\right)$, we will abuse the notation
and write $\phi$ for $\left[0,\phi\right]$.
\end{rem}

Finally, we recall the following lemma.
\begin{lem}
\cite[Proposition 2.4]{WahlRandal-Williams}\cite[Lemma 1.8]{soulieLMgeneralized}\label{lem:criterionfamilymorphismsfunctor}
Let $\mathscr{C}$ be a category and $F$ an object of $\mathbf{Fct}\left(\mathfrak{G},\mathscr{C}\right)$.
Assume that for $A,X,Y\in Obj\left(\mathfrak{G}\right)$, there exist
assignments $F\left(\left[X,id_{X\natural A}\right]\right):F\left(A\right)\rightarrow F\left(X\natural A\right)$
such that:
\begin{equation}
F\left(\left[Y,id_{Y\natural X\natural A}\right]\right)\circ F\left(\left[X,id_{X\natural A}\right]\right)=F\left(\left[Y\natural X,id_{Y\natural X\natural A}\right]\right).\label{eq:criterion}
\end{equation}
Then, the assignment $F\left(\left[X,g\right]\right)=F\left(g\right)\circ F\left(\left[X,id_{X\natural A}\right]\right)$
for $\left[X,g\right]\in Hom_{\mathfrak{U}\mathfrak{G}}\left(A,id_{X\natural A}\right)$
defines a functor $F:\mathfrak{U}\mathfrak{G}\rightarrow\mathscr{C}$
if and only if for all $A,X\in Obj\left(\mathfrak{G}\right)$, for
all $g''\in Aut_{\mathfrak{G}}\left(A\right)$ and all $g'\in Aut_{\mathfrak{G}}\left(X\right)$:
\begin{equation}
F\left(\left[X,id_{X\natural A}\right]\right)\circ F\left(g''\right)=F\left(g'\natural g''\right)\circ F\left(\left[X,id_{X\natural A}\right]\right).\label{eq:criterion'}
\end{equation}
\end{lem}

\paragraph{Pre-braided monoidal categories.}

Assuming that the strict monoidal groupoid $\left(\mathfrak{G},\natural,0\right)$
is braided, Quillen's bracket construction $\mathfrak{UG}$ also inherits
a strict monoidal structure (see Proposition \ref{prop:Quillen'sconstructionprebraided}).
Beforehand, we recall the notion of pre-braided category, introduced
by Randal-Williams and Wahl in \cite[Section 1]{WahlRandal-Williams}.
\begin{defn}
\cite[Definition 1.5]{WahlRandal-Williams}\label{def:defprebraided}
Let $\left(\mathfrak{C},\natural,0\right)$ be a strict monoidal category
such that the unit $0$ is initial. We say that the monoidal category
$\left(\mathfrak{C},\natural,0\right)$ is pre-braided if:

\begin{itemize}
\item The maximal subgroupoid $\mathscr{G}\mathfrak{r}\left(\mathfrak{C},\natural,0\right)$
 is a braided monoidal category, where the monoidal structure is induced
by that of $\left(\mathfrak{C},\natural,0\right)$.
\item For all objects $A$ and $B$ of $\mathfrak{C}$, the braiding associated
with the maximal subgroupoid $b_{A,B}^{\mathfrak{C}}:A\natural B\longrightarrow B\natural A$
satisfies:
\begin{equation}
b_{A,B}^{\mathfrak{C}}\circ\left(id_{A}\natural\iota_{B}\right)=\iota_{B}\natural id_{A}:A\longrightarrow B\natural A.\label{eq:defbraid}
\end{equation}
\end{itemize}
\end{defn}

\begin{rem}
A braided monoidal category is always pre-braided but the converse
is false. Indeed, for a pre-braided monoidal category, the opposite
of condition (\ref{eq:defbraid}) (i.e. $b_{A,B}^{\mathfrak{C}}\circ\left(\iota_{B}\natural id_{A}\right)=id_{A}\natural\iota_{B}$)
does not hold generally speaking, whereas this is a necessary property
for a braided monoidal category. For instance, the category $\left(\mathfrak{U\boldsymbol{\beta}},\natural,0\right)$
is pre-braided monoidal but not braided since $b_{1,2}^{\boldsymbol{\beta}}\circ\left(\iota_{1}\natural id{}_{2}\right)\neq id_{2}\natural\iota_{1}$
(see \cite[Remark 5.24]{WahlRandal-Williams} if more details are
required).
\end{rem}

Finally, we give the effect of Quillen's bracket construction over
the strict braided monoidal groupoid $\left(\mathfrak{G},\natural,0\right)$.
\begin{prop}
\cite[Proposition 1.8]{WahlRandal-Williams}\label{prop:Quillen'sconstructionprebraided}
Suppose that the strict monoidal groupoid $\left(\mathfrak{G},\natural,0\right)$
has no zero divisors and that $Aut_{\mathfrak{G}}(0)=\left\{ id_{0}\right\} $.
If the groupoid $\left(\mathfrak{G},\natural,0\right)$ is braided,
then the category $\left(\mathfrak{UG},\natural,0\right)$ is pre-braided
monoidal. If the groupoid $\left(\mathfrak{G},\natural,0\right)$
is symmetric, then the category $\left(\mathfrak{U}\mathfrak{G},\natural,0\right)$
is symmetric monoidal.

The monoidal structure on the category $\left(\mathfrak{UG},\natural,0\right)$
is defined on objects as for $\left(\mathfrak{G},\natural,0\right)$
and defined on morphisms by letting, for $\left[X,f\right]\in Hom_{\mathfrak{UG}}\left(A,B\right)$
and $\left[Y,g\right]\in Hom_{\mathfrak{UG}}\left(C,D\right)$:
\[
\left[X,f\right]\natural\left[Y,g\right]=\left[X\natural Y,\left(f\natural g\right)\circ\left(id_{X}\natural\left(b_{A,Y}^{\mathfrak{G}}\right)^{-1}\natural id_{C}\right)\right].
\]
In particular, the canonical functor $\mathfrak{G}\rightarrow\mathfrak{UG}$
(see Remark \ref{rem:canonicalfunctor}) is monoidal.
\end{prop}

\paragraph{Homogeneous categories.}

The notion of homogeneous category is introduced by Randal-Williams
and Wahl in \cite[Section 1]{WahlRandal-Williams}, inspired by the
set-up of Djament and Vespa in \cite[Section 1.2]{DV1}. With two
additional assumptions, Quillen's bracket construction $\mathfrak{UG}$
from a strict monoidal groupoid $\left(\mathfrak{G},\natural,0\right)$
is endowed with an homogeneous category structure. This type of category
is very useful to deal with homological stability with twisted coefficients
questions (see \cite{WahlRandal-Williams}) or to work on the stable
homology with twisted coefficient (see \cite{DV1}, \cite{DV2} and
Section \ref{sec:A-general-result}).

Let $\left(\mathfrak{C},\natural,0\right)$ be a small strict monoidal
category in which the unit $0$ is also initial. For all objects $A$
and $B$ of $\mathfrak{C}$, we consider the morphism $\iota_{A}\natural id_{B}:0\natural B\longrightarrow A\natural B$
and a set of morphisms:
\[
Fix_{A}\left(B\right)=\left\{ \phi\in Aut\left(A\natural B\right)\mid\phi\circ\left(\iota_{A}\natural id_{B}\right)=\iota_{A}\natural id_{B}\right\} .
\]
Since $\left(\mathfrak{C},\natural,0\right)$ is assumed to be small,
$Hom_{\mathfrak{C}}\left(A,B\right)$ is a set and $Aut_{\mathfrak{C}}(B)$
defines a group (with composition of morphisms as the group product).
The group $Aut_{\mathfrak{C}}(B)$ acts by post-composition on $Hom_{\mathfrak{C}}\left(A,B\right)$:
\[
\begin{array}{ccc}
Aut_{\mathfrak{C}}(B)\times Hom_{\mathfrak{C}}(A,B) & \longrightarrow & Hom_{\mathfrak{C}}(A,B).\\
\left(\phi,f\right) & \longmapsto & \phi\circ f
\end{array}
\]

\begin{defn}
\label{def:homogenousandlocallyhomogenous}Let $\left(\mathfrak{C},\natural,0\right)$
be a small strict monoidal category where the unit $0$ is initial.
We consider the following axioms:

\begin{itemize}
\item $\mathbf{\left(H1\right)}$: for all objects $A$ and $B$ of the
category $\mathfrak{C}$, the action by post-composition of $Aut_{\mathfrak{C}}(B)$
on $Hom_{\mathfrak{C}}\left(A,B\right)$ is transitive.
\item $\mathbf{\left(H2\right)}$: for all objects $A$ and $B$ of the
category $\mathfrak{C}$, the map $Aut_{\mathfrak{C}}\left(A\right)\rightarrow Aut_{\mathfrak{C}}\left(A\natural B\right)$
sending $f\in Aut_{\mathfrak{C}}\left(A\right)$ to $f\natural id_{B}$
is injective with image $Fix_{A}\left(B\right)$.
\end{itemize}
The category $\left(\mathfrak{C},\natural,0\right)$ is homogeneous
if it satisfies the axioms $\mathbf{\left(H1\right)}$ and $\mathbf{\left(H2\right)}$.
\end{defn}

As a consequence of the axioms $\mathbf{\left(H1\right)}$ and $\mathbf{\left(H2\right)}$,
we deduce that:
\begin{lem}
\label{rem:Homogenousgroupset}If $\left(\mathfrak{C},\natural,0\right)$
is a homogeneous category, then $Hom_{\mathfrak{C}}\left(B,A\natural B\right)\cong Aut_{\mathfrak{C}}\left(A\natural B\right)/Aut_{\mathfrak{C}}\left(A\right)$
for all objects $A$ and $B$ and where $Aut_{\mathfrak{C}}\left(A\right)$
acts on $Aut_{\mathfrak{C}}\left(A\natural B\right)$ by precomposition.
\end{lem}

We now give the two additional properties so that if a strict monoidal
groupoid $\left(\mathfrak{G},\natural,0\right)$ satisfy them, then
Quillen's bracket construction $\mathfrak{UG}$ is homogeneous.
\begin{defn}
\label{def:conditionscancelandinject}Let $\left(\mathfrak{C},\natural,0\right)$
be a strict monoidal category. We define two assumptions.

\begin{itemize}
\item $\mathbf{\left(C\right)}$: for all objects $A$, $B$ and $C$ of
$\mathfrak{C}$, if $A\natural C\cong B\natural C$ then $A\cong B$.
\item $\mathbf{\left(I\right)}$: for all objects $A$, $B$ of $\mathfrak{C}$,
the morphism $Aut_{\mathfrak{C}}\left(A\right)\rightarrow Aut_{\mathfrak{C}}\left(A\natural B\right)$
sending $f\in Aut_{\mathfrak{C}}\left(A\right)$ to $f\natural id_{B}$
is injective.
\end{itemize}
\end{defn}

\begin{thm}
\cite[Theorem 1.10]{WahlRandal-Williams}\label{thm:consequencecondcancalandinject}
Let $\left(\mathfrak{G},\natural,0\right)$ be a braided monoidal
groupoid with no zero divisors. If the groupoid $\mathfrak{G}$ satisfies
$\mathbf{\left(C\right)}$ and $\mathbf{\left(I\right)}$, then $\mathfrak{UG}$
is homogeneous.
\end{thm}

\section{Twisted stable homologies of semidirect products\label{sec:sssp}}

This section introduces a general method to compute the stable homology
with twisted coefficients using semidirect product structures arising
naturally from the families of mapping class groups. We first establish
the general result of Corollary \ref{cor:bigcor} for the homology
of semidirect products with twisted coefficients. These results are
finally applied in Section \ref{subsec:Applications} to compute explicitly
some homology groups with twisted coefficients for mapping class groups
of orientable surfaces and automorphisms of free groups with boundaries.
Beforehand, we take this opportunity to introduce the following terminology:
\begin{defn}
\label{def:A-family-of group}A family of groups is a functor $\mathbf{G}_{-}:\left(\mathbb{N},\leq\right)\longrightarrow\mathfrak{Gr}$
such that for all natural numbers $n$, $\mathbf{G}_{-}\left(\gamma{}_{n}\right):\mathbf{G}_{n}\hookrightarrow\mathbf{G}_{n+1}$
is an injective group morphism.
\end{defn}

\subsection{A general result for the homology of semidirect products\label{subsec:frameworksemidirect}}

First, we present some properties for the homology with twisted coefficients
for a semidirect product and prove the general statement of Corollary
\ref{cor:bigcor}.

Let $\mathcal{Q}$ be a groupoid with objects indexed by the natural
numbers. An object of $\mathcal{Q}$ is thus denoted by $\underline{n}$,
where $n$ is its corresponding indexing natural number. We denote
by $Aut_{\mathcal{Q}}\left(\underline{n}\right)=Q_{n}$ the automorphism
groups and assume that there are no morphisms between distinct objects.
We assume that there exists a family of groups $K_{-}:\left(\mathbb{N},\leq\right)\rightarrow\mathfrak{Gr}$
and a functor $\mathcal{A}_{\mathcal{Q}}:\mathcal{Q}\rightarrow\mathfrak{Gr}$
such that $K_{n}$ is any free goup and $\mathcal{A}_{\mathcal{Q}}\left(\underline{n}\right)=K{}_{n}$
for all natural numbers $n$. We denote by $\mathcal{A}_{\mathcal{Q},n}:Q_{n}\rightarrow Aut_{\mathfrak{Gr}}\left(K_{n}\right)$
the group morphism induced by the functor $\mathcal{A}_{\mathcal{Q}}$
for each natural number $n$. Hence, we form the split short exact
sequence:
\begin{equation}
\xymatrix{1\ar@{->}[r] & K_{n}\ar@{->}[r]^{k_{n}\,\,\,\,\,\,\,\,\,\,\,} & K_{n}\underset{\mathcal{A}_{\mathcal{Q},n}}{\rtimes}Q_{n}\ar@{->}[r]^{\,\,\,\,\,\,\,\,\,\,q_{n}} & Q_{n}\ar@{->}[r] & 1}
\label{eq:semidirect}
\end{equation}
and we denote by $s_{n}:Q_{n}\rightarrow K_{n}\underset{\mathcal{A}_{\mathcal{Q},n}}{\rtimes}Q_{n}$
the canonical splitting of $q_{n}$. For all natural numbers $n$,
we fix $M_{n}$ a $R\left[K_{n}\underset{\mathcal{A}_{\mathcal{Q},n}}{\rtimes}Q_{n}\right]$-module.
We abuse the notation and write $M_{n}$ for the restriction of $M_{n}$
from $K_{n}\underset{\mathcal{A}_{\mathcal{Q},n}}{\rtimes}Q_{n}$
to $K_{n}$. Then we can construct the following long exact sequcence
in homology, analogous to the Gysin sequence for homology:
\begin{prop}
\label{thm:bigthm3}For $M_{n}$ an $R\left[K_{n}\underset{\mathcal{A}_{\mathcal{Q},n}}{\rtimes}Q_{n}\right]$-module,
the short exact sequence (\ref{eq:semidirect}) induces a long exact
sequence:
\begin{equation}
\xymatrix{\cdots\ar@{->}[r] & H_{*+1}\left(Q_{n},H_{0}\left(K_{n},M_{n}\right)\right)\ar@{->}[d]^{d_{*+1,0}^{2}}\\
 & H_{*-1}\left(Q_{n},H_{1}\left(K_{n},M_{n}\right)\right)\ar@{->}[r]^{\,\,\,\,\,\,\varphi_{*}} & H_{*}\left(K_{n}\underset{\mathcal{A}_{\mathcal{Q},n}}{\rtimes}Q_{n},M_{n}\right)\ar@{->}[r]^{\psi_{*}} & H_{*}\left(Q_{n},H_{0}\left(K_{n},M_{n}\right)\right)\ar@{->}[d]^{d_{*,0}^{2}}\\
 &  &  & H_{*-2}\left(Q_{n},H_{1}\left(K_{n},M_{n}\right)\right)\ar@{->}[r] & \cdots
}
\label{eq:les}
\end{equation}
where $\left\{ d_{p,q}^{2}\right\} _{p,q\in\mathbb{N}}$ denote the
differentials of the second page of the Lyndon-Hochschild-Serre spectral
sequence associated with  the short exact sequence (\ref{eq:semidirect}).
\end{prop}

\begin{proof}
Applying the Lyndon-Hochschild-Serre spectral sequence to the short
exact sequence (\ref{eq:semidirect}), we obtain the following convergent
first quadrant spectral sequence:
\begin{equation}
E_{pq}^{2}:H_{p}\left(Q_{n},H_{q}\left(K_{n},M_{n}\right)\right)\Longrightarrow H_{p+q}\left(K_{n}\underset{\mathcal{A}_{\mathcal{Q},n}}{\rtimes}Q_{n},M_{n}\right).\label{eq:ssGnHn}
\end{equation}
Since $K_{n}$ is a free group, $H_{q}\left(K_{n},M_{n}\right)=0$
for $q\geq2$. The result is a classical consequence of the fact that
the spectral sequence (\ref{eq:ssGnHn}) has only two rows. In particular,
the map $\varphi_{*}$ is defined by the composition:
\[
H_{*-1}\left(Q_{n},H_{1}\left(K_{n},M_{n}\right)\right)\twoheadrightarrow H_{*-1}\left(Q_{n},H_{1}\left(K_{n},M_{n}\right)\right)/\textrm{Im}\left(d_{*+1,0}^{2}\right)\hookrightarrow H_{*}\left(K_{n}\underset{\mathcal{A}_{\mathcal{Q},n}}{\rtimes}Q_{n},M_{n}\right);
\]
the map $\psi_{*}$ is the coinflation map $\textrm{Coinf}_{K_{n}\underset{\mathcal{A}_{\mathcal{Q},n}}{\rtimes}Q_{n}}^{Q_{n}}\left(M_{n}\right)$,
induced by the composition:
\[
H_{*}\left(K_{n}\underset{\mathcal{A}_{\mathcal{Q},n}}{\rtimes}Q_{n},M_{n}\right)\twoheadrightarrow Ker\left(d_{*,0}^{2}\right)\hookrightarrow H_{*}\left(Q_{n},H_{0}\left(K_{n},M_{n}\right)\right).
\]
\end{proof}
\begin{cor}
\label{cor:bigcor} Let $n$ be a natural number. Assume that the
free group $K_{n}$ acts trivially on the $R$-module $M_{n}$. Then,
for all natural numbers $q\geq1$:
\begin{equation}
H_{q-1}\left(Q_{n},H_{1}\left(K_{n},R\right)\underset{R}{\otimes}M_{n}\right)\oplus H_{q}\left(Q_{n},M_{n}\right)\cong H_{q}\left(K_{n}\underset{\mathcal{A}_{\mathcal{Q},n}}{\rtimes}Q_{n},M_{n}\right).\label{eq:splitsslhS}
\end{equation}
\end{cor}

\begin{proof}
As $M_{n}$ is a trivial $K_{n}$-module:
\[
H_{1}\left(K_{n},M_{n}\right)\cong H_{1}\left(K_{n},R\right)\underset{R}{\otimes}M_{n}\,\textrm{ and }\,H_{0}\left(K_{n},M_{n}\right)\cong M_{n},
\]
and the coinflation map $\psi_{*}=\textrm{Coinf}_{K_{n}\underset{\mathcal{A}_{\mathcal{Q},n}}{\rtimes}Q_{n}}^{Q_{n}}\left(M_{n}\right)$
is equal to the corestriction map $\textrm{Cores}_{K_{n}\underset{\mathcal{A}_{\mathcal{Q},n}}{\rtimes}Q_{n}}^{Q_{n}}\left(M_{n}\right)$.
Hence, denoting by $H_{*}\left(q_{n},M_{n}\right)$ the map induced
in homology by $q_{n}:K_{n}\underset{\mathcal{A}_{\mathcal{Q},n}}{\rtimes}Q_{n}\rightarrow Q_{n}$,
we deduce that $\psi_{*}=H_{*}\left(q_{n},M_{n}\right)$. By the functoriality
of group homology, the splitting $s_{n}:Q_{n}\rightarrow K_{n}\underset{\mathcal{A}_{\mathcal{Q},n}}{\rtimes}Q_{n}$
of $q_{n}$ induces a splitting in homology $H_{*}\left(s_{n},M_{n}\right)$
of $H_{*}\left(q_{n},M_{n}\right)$. Hence, $H_{*}\left(p_{n},M_{n}\right)$
is an epimorphism and a fortiori $Ker\left(d_{*,0}^{2}\right)\cong H_{*}\left(Q_{n},M_{n}\right)$.
Therefore, $d_{*,0}^{2}=0$ and the exact sequence (\ref{eq:les})
gives a split short exact sequence of abelian groups for every $q\geq1$:
\begin{equation}
\xymatrix{1\ar@{->}[r] & H_{q-1}\left(Q_{n},H_{1}\left(K_{n},R\right)\underset{R}{\otimes}M_{n}\right)\ar@{->}[r]^{\,\,\,\,\,\,\,\,\,\,\,\varphi_{q}} & H_{q}\left(K_{n}\underset{\mathcal{A}_{\mathcal{Q},n}}{\rtimes}Q_{n},M_{n}\right)\ar@{->}[rr]^{\,\,\,\,\,\,\,\,\,\,H_{q}\left(q_{n},M_{n}\right)} &  & H_{q}\left(Q_{n},M_{n}\right)\ar@{->}[r] & 1.}
\label{eq:profsplisslhS}
\end{equation}
\end{proof}

\subsection{Properties of the twisted coefficients}

Our aim here is to study the twisted coefficients $H_{1}\left(K_{n},R\right)\underset{R}{\otimes}M_{n}$
appearing in Corollary \ref{cor:bigcor} so as to prove Lemma \ref{lem:H1tensorm}.
This last result will be useful to prove Theorem \ref{thm:resultAutFnwithbound}.
We now make the following further assumptions that:

\begin{itemize}
\item the groupoid $\mathcal{Q}$ is a braided strict monoidal category
(we denote by $\left(\mathcal{Q},\natural,0\right)$ the monoidal
structure);
\item there exists a free group $K$ such that $K_{n}\cong K^{*n}$ and
that that $K_{-}\left(\gamma_{n}\right)=\iota_{K}*id_{K_{n}}$ (where
$\gamma{}_{n}$ denotes the unique element of $Aut_{\left(\mathbb{N},\leq\right)}\left(n,n+1\right)$)
for all natural numbers $n$;
\item the functor $\mathcal{A}_{\mathcal{Q}}$ takes values in the subcategory
$\mathfrak{gr}\subset\mathfrak{Gr}$ and defines a strict monoidal
functor $\left(\mathcal{Q},\natural,0\right)\rightarrow\left(\mathfrak{gr},*,0\right)$.
\end{itemize}
These assumptions allow to define the functor $K_{-}$ on the category
$\mathfrak{U}\mathcal{Q}$:
\begin{lem}
\label{lem:ExtensionfunctorUQ}Assigning $\mathcal{A}_{\mathcal{Q}}\left(\left[\underline{1},id_{\underline{n+1}}\right]\right)=K_{-}\left(\gamma_{n}\right)$
for all natural numbers $n$, we define a functor $\mathcal{A}_{\mathcal{Q}}:\mathfrak{U}\mathcal{Q}\rightarrow\mathfrak{gr}$.
\end{lem}

\begin{proof}
We use Lemma \ref{lem:criterionfamilymorphismsfunctor} to prove this
result: namely, we show that relations (\ref{eq:criterion}) and (\ref{eq:criterion'})
of this lemma are satisfied. It follows from the fact that $K_{-}$
is a functor on $\left(\mathbb{N},\leq\right)$, that the relation
(\ref{eq:criterion}) of Lemma \ref{lem:criterionfamilymorphismsfunctor}
is satisfied by $\mathcal{A}_{\mathcal{Q}}$. Let $n$ and $n'$ be
natural numbers such that $n'\geq n$, let $q\in Q_{n}$ and $q'\in Q_{n'}$.
We denote by $e_{K_{n'}}$ the neutral element of $K_{n'}$. Since
$\mathcal{A}_{\mathcal{Q}}$ is monoidal, we compute for all $k\in K_{n}$:
\begin{eqnarray*}
\left(\mathcal{A}_{\mathcal{Q}}\left(q'\natural q\right)\circ\mathcal{A}_{\mathcal{Q}}\left(\left[\underline{n'},id_{\underline{n'+n}}\right]\right)\right)\left(k\right) & = & \left(\mathcal{A}_{\mathcal{Q}}\left(q'\right)*\mathcal{A}_{\mathcal{Q}}\left(q\right)\right)\left(e_{K_{n'}}*k\right)\\
 & = & e_{K_{n'}}*\mathcal{A}_{\mathcal{Q}}\left(q\right)\left(k\right)\\
 & = & \left(\mathcal{A}_{\mathcal{Q}}\left(\left[\underline{n'},id_{\underline{n'+n}}\right]\right)\circ\mathcal{A}_{\mathcal{Q}}\left(q\right)\right)\left(k\right).
\end{eqnarray*}
Hence, the relation (\ref{eq:criterion'}) of Lemma \ref{lem:criterionfamilymorphismsfunctor}
is satisfied by $\mathcal{A}_{\mathcal{Q}}$.
\end{proof}

\paragraph{Recollections on strong polynomial functors.}

We deal here with the concept of strong and very strong polynomial
functors, which will be useful to prove Theorems \ref{thm:homologymcgtensorproduct}
and \ref{thm:resultAutFnwithbound}. We refer the reader to \cite[Section 3]{soulieLMBilan}
for a complete introduction to these notions for pre-braided monoidal
categories as source category, extending the previous framework due
to Djament and Vespa in \cite{DV3} for symmetric monoidal categories.
They also are particular case of coefficient systems of finite degree
introduced by Randal-Williams and Wahl in \cite{WahlRandal-Williams},
thus providing a natural setting to study homological stability.

\textbf{From now we fix $\left(\mathfrak{M},\natural,0\right)$ a
pre-braided strict monoidal category such that the monoidal unit $0$
is an initial object.} For all objects $X$ of $\mathfrak{M}$, the
monoidal structure $\natural$ defines the endofunctor $X\natural-:\mathfrak{M}\rightarrow\mathfrak{M}$,
which sends the object $Y$ to the object $X\natural Y$. We define
the translation functor $\tau_{X}:\mathbf{Fct}\left(\mathfrak{M},R\textrm{-}\mathfrak{Mod}\right)\rightarrow\mathbf{Fct}\left(\mathfrak{M},R\textrm{-}\mathfrak{Mod}\right)$
to be the endofunctor obtained by precomposition by $X\natural-$.

For all objects $F$ of $\mathbf{Fct}\left(\mathfrak{M},R\textrm{-}\mathfrak{Mod}\right)$
and all objects $X$ of $\mathfrak{M}$, we denote by $i_{X}\left(F\right):F\rightarrow\tau_{X}\left(F\right)$
the natural transformation induced by the unique morphism $\left[X,id_{X}\right]:0\rightarrow X$
of $\mathfrak{M}$. This induces $i_{X}:Id_{\mathbf{Fct}\left(\mathfrak{M},R\textrm{-}\mathfrak{Mod}\right)}\rightarrow\tau_{X}$
a natural transformation of $\mathbf{Fct}\left(\mathfrak{M},R\textrm{-}\mathfrak{Mod}\right)$.
Since the category $\mathbf{Fct}\left(\mathfrak{M},R\textrm{-}\mathfrak{Mod}\right)$
is abelian (since the target category $R\textrm{-}\mathfrak{Mod}$
is abelian), the kernel and cokernel of the natural transformation
$i_{X}$ exist. We define the functors $\kappa_{X}=\ker\left(i_{X}\right)$
and $\delta_{X}=\textrm{coker}\left(i_{X}\right)$. Then:
\begin{defn}
We recursively define on $d\in\mathbb{N}$ the categories $\mathcal{P}ol_{d}^{strong}\left(\mathfrak{M},R\textrm{-}\mathfrak{Mod}\right)$
and $\mathcal{VP}ol_{d}\left(\mathfrak{M},R\textrm{-}\mathfrak{Mod}\right)$
of strong and very strong polynomial functors of degree less than
or equal to $d$ to be the full subcategories of $\mathbf{Fct}\left(\mathfrak{M},R\textrm{-}\mathfrak{Mod}\right)$
as follows:

\begin{enumerate}
\item If $d<0$, $\mathcal{P}ol_{d}^{strong}\left(\mathfrak{M},R\textrm{-}\mathfrak{Mod}\right)=\mathcal{VP}ol_{d}\left(\mathfrak{M},R\textrm{-}\mathfrak{Mod}\right)=\left\{ 0\right\} $;
\item if $d\geq0$, the objects of $\mathcal{P}ol_{d}^{strong}\left(\mathfrak{M},R\textrm{-}\mathfrak{Mod}\right)$
are the functors $F$ such that the functor $\delta_{X}\left(F\right)$
is an object of $\mathcal{P}ol_{d-1}^{strong}\left(\mathfrak{M},R\textrm{-}\mathfrak{Mod}\right)$
for all objects $X$ of $\mathfrak{M}$; the objects of $\mathcal{VP}ol_{d}\left(\mathfrak{M},R\textrm{-}\mathfrak{Mod}\right)$
are the objects $F$ of $\mathcal{P}ol_{d}\left(\mathfrak{M},R\textrm{-}\mathfrak{Mod}\right)$
such that $\kappa_{X}\left(F\right)=0$ and the functor $\delta_{X}\left(F\right)$
is an object of $\mathcal{\mathcal{VP}}ol_{d-1}\left(\mathfrak{M},R\textrm{-}\mathfrak{Mod}\right)$
for all objects $X$ of $\mathfrak{M}$.
\end{enumerate}
For an object $F$ of $\mathbf{Fct}\left(\mathfrak{M},R\textrm{-}\mathfrak{Mod}\right)$
which is strong of degree less than or equal to $n\in\mathbb{N}$,
the smallest natural number $d\leq n$ for which $F$ is an object
of $\mathcal{P}ol_{d}^{strong}\left(\mathfrak{M},R\textrm{-}\mathfrak{Mod}\right)$
is called the strong degree of $F$. If in addition $F$ is strong
very strong polynomial, its strong degree is also the smallest natural
number $d\leq n$ for which $F$ is an object of $\mathcal{VP}ol_{d}\left(\mathfrak{M},R\textrm{-}\mathfrak{Mod}\right)$
and is then also called the very strong degree of $F$.
\end{defn}

\begin{rem}
\label{rem:caseunitterminal}If in addition the unit $0$ of the monoidal
structure $\left(\mathfrak{M},\natural,0\right)$ is a terminal object,
the evanescence functors automatically vanish and a fortiori the notions
of strong and very strong polynomial functors are equivalent.
\end{rem}

Furthermore, we have the following property which will be useful for
our further work. The result for strong polynomial functors is already
established in \cite[Proposition 3.8]{soulieLMBilan} and previously
for symmetric monoidal categories \cite[Proposition 1.7]{DV3}.
\begin{prop}
\label{prop:precompositionstrongmono}Let $\mathfrak{M}'$ be another
pre-braided strict monoidal category such that such that its monoidal
unit is an initial object. Let $\alpha:\mathfrak{M}\rightarrow\mathfrak{M}'$
be a strong monoidal functor. Then, the precomposition by $\alpha$
provide functors $\mathcal{VP}ol_{n}\left(\mathfrak{M}',R\textrm{-}\mathfrak{Mod}\right)\rightarrow\mathcal{VP}ol_{n}\left(\mathfrak{M},R\textrm{-}\mathfrak{Mod}\right)$
which preserves the degree of polynomiality.
\end{prop}

\begin{proof}
Let $X$ an object of $\mathfrak{M}$ and $F$ be an object of $\mathbf{Fct}\left(\mathfrak{M},R\textrm{-}\mathfrak{Mod}\right)$.
Since $\alpha$ is strong monoidal, we deduce that there is a natural
equivalence $\tau_{X}\left(F\circ\alpha\right)\cong\left(\tau_{\alpha\left(X\right)}F\right)\circ\alpha$.
It is a standard fact that the precomposition by $\alpha$ is an exact
functor. Then the universal properties of the kernel and cokernel
imply that there are natural equivalences $\delta_{X}\left(F\circ\alpha\right)\cong\left(\delta_{\alpha\left(X\right)}F\right)\circ\alpha$
and $\kappa_{X}\left(F\circ\alpha\right)\cong\left(\kappa_{\alpha\left(X\right)}F\right)\circ\alpha$.
The results then follow from a straightforward recursion on the degree
of polynomiality.
\end{proof}

\paragraph{First homology functor.}

Recall that the homology group $H_{1}\left(-,R\right)$ defines a
functor from the category $\mathfrak{Gr}$ to the category $R\textrm{-}\mathfrak{Mod}$
(see for example \cite[Section 8]{brown}). Hence, we introduce the
following functor: 
\begin{defn}
\label{prop:H1functoronG-1} The homology groups $\left\{ H_{1}\left(K_{n},R\right)\right\} _{n\in\mathbb{N}}$
assemble to define a functor $H_{1}\left(\mathcal{A}_{\mathcal{Q}},R\right):\mathfrak{U}\mathcal{Q}\rightarrow R\textrm{-}\mathfrak{Mod}$
by the composite $H_{1}\left(-,R\right)\circ\mathcal{A}_{\mathcal{Q}}$.
It is called the first homology functor of $\mathcal{A}_{\mathcal{Q}}$.
\end{defn}

If $R=\mathbb{Z}$ and $K_{n}$ is finitely generated for all natural
numbers $n$, the target category of $H_{1}\left(\mathcal{A}_{\mathcal{Q}},R\right)$
is the full subcategory $\mathbf{ab}\subset\mathbf{Ab}$ of finitely
generated abelian groups. Let $m$ be a natural number. We then define
a functor $H_{1}\left(\mathcal{A}_{\mathcal{Q}},\mathbb{Z}\right)^{\otimes m}:\mathfrak{U}\mathcal{Q}\rightarrow\mathbf{ab}$
by the composite $-^{\otimes m}\circ H_{1}\left(\mathcal{A}_{\mathcal{Q}},\mathbb{Z}\right)$
where $-^{\otimes m}:\mathbf{ab}\rightarrow\mathbf{ab}$ sends an
object $G$ to $G^{\otimes m}$.
\begin{lem}
\label{lem:H1tensorm}If the groups $K_{n}$ are finitely generated
free for all $n$, then the functor $H_{1}\left(\mathcal{A}_{\mathcal{Q}},\mathbb{Z}\right)^{\otimes m}$
is very strong polynomial of degree $m$.
\end{lem}

\begin{proof}
Recall that the free product gives a symmetric monoidal structure
$\left(\mathfrak{Gr},*,0_{\mathfrak{Gr}}\right)$ (which restricts
to $\mathfrak{gr}$), that the direct sum defines a symmetric monoidal
structure $\left(\mathbf{Ab},\oplus,0_{\mathfrak{Gr}}\right)$ (which
restricts to $\mathbf{ab}$) and that symmetric monoidal categories
are particular cases of pre-braided monoidal ones. By the result of
the homology of a free product of groups (see for example \cite[Corollary 6.2.10]{Weibel1}),
the first homology group is a strong monoidal functor $H_{1}\left(-,R\right):\left(\mathfrak{Gr},*,0_{\mathfrak{Gr}}\right)\rightarrow\left(\mathbf{Ab},\oplus,0_{\mathfrak{Gr}}\right)$
and a fortiori so is the restriction $H_{1}\left(-,R\right):\mathfrak{gr}\rightarrow\mathbf{ab}$.
As $\mathcal{A}_{\mathcal{Q}}$ is a strict monoidal functor, then
$H_{1}\left(\mathcal{A}_{\mathcal{Q}},\mathbb{Z}\right)$ is a strong
monoidal functor. This is a well-known fact that the $m$-th tensor
power functor $-^{\otimes m}:\mathbf{ab}\rightarrow\mathbf{ab}$ is
very strong polynomial of degree $m$ (see \cite[Appendice A]{DV1}
for example). Then the result follows from Proposition \ref{prop:precompositionstrongmono}.
\end{proof}

\paragraph{Pointwise tensor product.}

We finally recall the following result, used to prove Theorem \ref{thm:resultAutFnwithbound}.
For $\mathfrak{M}$ a pre-braided strict monoidal category such that
such that its monoidal unit is an initial object, the pointwise tensor
product of two objects of $\mathbf{Fct}\left(\mathfrak{M},R\textrm{-}\mathfrak{Mod}\right)$
defines an object of $\mathbf{Fct}\left(\mathfrak{M},R\textrm{-}\mathfrak{Mod}\right)$,
assigning
\[
\left(M\underset{R}{\otimes}M'\right)\left(X\right)=M\left(X\right)\underset{R}{\otimes}M'\left(X\right)
\]
for $M,M'\in\mathbf{Fct}\left(\mathfrak{M},R\textrm{-}\mathfrak{Mod}\right)$
and for all objects $X$ of $\mathfrak{M}$.
\begin{lem}
\label{cor:increasepoly} If $M$ and $M'$ are strong polynomial
functors, then $M\underset{R}{\otimes}M'$ is a strong polynomial
functor.
\end{lem}

\begin{proof}
We fix an object $X$ of $\mathfrak{M}$. Since the translation functor
$\tau_{X}$ commutes with all limits, and as a colimit of a natural
transformation between $Id$ and $\tau_{X}$, the functor $\delta_{X}$
commutes with the (pointwise) product $\times$. Let $d$ be the largest
of the two strong polynomial degrees. Hence, $\underset{d+1\textrm{ times}}{\underbrace{\delta_{X}\cdots\delta_{X}}}\left(M\times M'\right)=0$
and therefore $M\underset{R}{\otimes}M'$ is a strong polynomial functor
of degree less than or equal to $d+1$.
\end{proof}

\subsection{Applications\label{subsec:Applications}}

Many families of mapping class groups fall within the framework of
Section \ref{subsec:frameworksemidirect}. Corollary \ref{cor:bigcor}
is the key result to compute the homology with twisted coefficients
for these families of groups.

\subsubsection{Mapping class groups of orientable surfaces\label{subsec:FirstsectionMapping-class-groups}}

Let $\varSigma_{g,i}^{s}$ denote a smooth compact connected orientable
surface with (orientable) genus $g\in\mathbb{N}$, $s\in\mathbb{N}$
marked points and $i\in\left\{ 1,2\right\} $ boundary components,
with $I:\left[-1,1\right]\rightarrow\partial\varSigma_{g,i}^{s}$
a parametrized interval in the boundary if $i=1$ and $p=0\in I$
a basepoint. We denote by $\mathbf{\Gamma}_{g,1}^{s}$ (respectively
$\mathbf{\Gamma}_{g,1}^{\left[s\right]}$) the isotopy classes of
diffeomorphisms of $\varSigma_{g,1}^{s}$ preserving the orientation,
restricting to the identity on a neighbourhood of the parametrized
interval $I$ and permuting (respectively fixing) the marked points.
If $s=0$, we omit it from the notation. Recall that, up to isotopy,
fixing the interval $I$ is the same as fixing the whole boundary
component pointwise. When there is no ambiguity, we omit the parametrized
interval $I$ from the notation.

Let $\natural$ be the boundary connected sum along half of each interval
$I$. For two decorated surfaces $\varSigma_{g_{1},1}^{s_{1}}$ and
$\varSigma_{g_{2},1}^{s_{2}}$, the boundary connected sum $\varSigma_{g_{1},1}^{s_{1}}\natural\varSigma_{g_{2},1}^{s_{2}}$
is defined as the surface obtained from gluing $\varSigma_{g_{1},1}^{s_{1}}$
and $\varSigma_{g_{2},1}^{s_{2}}$ along the half-interval $I_{1}^{+}$
and the half-interval $I_{2}^{-}$, and $I_{1}\natural I_{2}=I_{1}^{-}\bigcup I_{2}^{+}$.
The homeomorphisms being the identity on a neighbourhood of the parametrized
intervals $I_{1}$ and $I_{2}$, we canonically extend the diffeomorphisms
of $\varSigma_{g_{1},1}^{s_{1}}$ and $\varSigma_{g_{2},1}^{s_{2}}$
to $\varSigma_{g_{1},1}^{s_{1}}\natural\varSigma_{g_{2},1}^{s_{2}}$.
For completeness, we refer to \cite[Section 5.6.1]{WahlRandal-Williams},
for technical details.

We denote by $\mathbf{\Gamma}_{g,2}$ the isotopy classes of diffeomorphisms
of $\varSigma_{g,2}^{0}$ preserving the orientation and fixing the
boundary components pointwise. Recall that $R$ is a commutative ring
and we assume that the various mapping class groups act trivially
on it.

The following result is an essential tool for our work:
\begin{thm}
\label{thm:Birmansesorient}\cite{BirmanbraidslinksMCG} Let $g\geq1$,
$s\geq0$ be natural numbers and $x$ be a marked point in the interior
of $\varSigma_{g,1}^{s}$. Forgetting $x$ induces a map $\omega_{s}:\mathbf{\Gamma}_{g,1}^{\left[s+1\right]}\rightarrow\mathbf{\Gamma}_{g,1}^{\left[s\right]}$
which defines the following short exact sequence:

\begin{equation}
\xymatrix{1\ar@{->}[r] & \pi_{1}\left(\varSigma_{g,1}^{s},x\right)\ar@{->}[r] & \mathbf{\Gamma}_{g,1}^{\left[s+1\right]}\ar@{->}[r]^{\omega_{s}} & \mathbf{\Gamma}_{g,1}^{\left[s\right]}\ar@{->}[r] & 1.}
\label{eq:Birman1}
\end{equation}
Gluing a disc with a marked point $\varSigma_{0,1}^{1}$ on the boundary
component without the interval $I$ induces the following short exact
sequence:
\begin{equation}
\xymatrix{1\ar@{->}[r] & \mathbb{Z}\ar@{->}[r] & \mathbf{\Gamma}_{g,2}\ar@{->}[r]^{\rho} & \mathbf{\Gamma}_{g,1}^{1}\ar@{->}[r] & 1.}
\label{eq:Birman2}
\end{equation}
\end{thm}

For all natural numbers $g$ and $s$, we denote by $a_{\varSigma_{g,1}^{s}}^{x}$
the action of the mapping class group $\mathbf{\Gamma}_{g,1}^{\left[s\right]}$
on the fundamental group $\pi_{1}\left(\varSigma_{g,1}^{s},x\right)$.
\begin{lem}
\label{lem:splitting}The short exact sequence (\ref{eq:Birman1})
splits.
\end{lem}

\begin{proof}
We denote by $\textrm{Diff}^{\textrm{\ensuremath{\partial_{0}},points}}\left(\varSigma_{g,1}^{n}\right)$
the space of diffeomorphisms of the surface $\varSigma_{g,1}^{n}$
which fix the boundary pointwise and fix the marked points. We recall
that the exact sequence (\ref{eq:Birman1}) is constructed from the
long exact sequence of homotopy groups with the locally trivial fibration
$\textrm{Diff}^{\textrm{\ensuremath{\partial_{0}},points}}\left(\varSigma_{g,1}^{1+s}\right)\overset{\widehat{\omega}_{s}}{\hookrightarrow}\textrm{Diff}^{\textrm{\ensuremath{\partial_{0}},points}}\left(\varSigma_{g,1}^{s}\right)\rightarrow\varSigma_{g,1}\setminus\left\{ s\textrm{ points}\right\} $
using the fact that $\pi_{1}\left(\textrm{Diff}^{\textrm{\ensuremath{\partial_{0}},points}}\left(\varSigma_{g,1}^{s}\right)\right)=0$
by \cite[Théorème 1]{gramain1973type}. Namely the fibre $\widehat{\omega}_{s}$
is defined by forgetting that the additional marked point is fixed
and induces the morphism $\omega_{s}:\mathbf{\Gamma}_{g,1}^{\left[s+1\right]}\rightarrow\mathbf{\Gamma}_{g,1}^{\left[s\right]}$.

We consider $\textrm{Emb}\left(\varSigma_{0,1}^{1},\varSigma_{0,1}^{1}\natural\varSigma_{g,1}^{s}\right)$
the space of embeddings taking $I_{\varSigma_{0,1}^{1}}^{-}$ to $I_{\varSigma_{0,1}^{1}\natural\varSigma_{g,1}^{s}}^{-}$
and such that the complement of $\varSigma_{0,1}^{1}$ in $\varSigma_{0,1}^{1}\natural\varSigma_{g,1}^{s}\simeq\varSigma_{g,1}^{1+s}$
is diffeomorphic to $\varSigma_{g,1}^{s}$. Using the parameterised
isotopy extension theorem \cite[II 2.2.2 Corollaire 2]{Cerffibration},
there is a fibration sequence 
\[
\textrm{Diff}^{\textrm{\ensuremath{\partial_{0}},points}}\left(\varSigma_{g,1}^{s}\right)\overset{\varrho_{s}}{\hookrightarrow}\textrm{Diff}^{\textrm{\ensuremath{\partial_{0}},points}}\left(\varSigma_{0,1}^{1}\natural\varSigma_{g,1}^{s}\right)\rightarrow\textrm{Emb}\left(\left(\varSigma_{0,1}^{1}\right),\left(\varSigma_{0,1}^{1}\natural\varSigma_{g,1}^{s}\right)\right),
\]
which long exact sequence of homotopy groups defines a morphism $\pi_{0}\left(\varrho_{s}\right):\mathbf{\Gamma}_{g,1}^{\left[s\right]}\rightarrow\mathbf{\Gamma}_{g,1}^{\left[s+1\right]}$.
More precisely, for all $\varphi\in\textrm{Diff}^{\textrm{\ensuremath{\partial_{0}},points}}\left(\varSigma_{g,1}^{s}\right)$,
the morphism $\varrho_{s}$ is explicitly defined by $\varrho_{s}\left(\varphi\right)=id_{\varSigma_{0,1}^{1}}\natural\varphi$.
Implicitly, we identify $\varSigma_{g,1}^{s}$ with $\varSigma_{0,1}^{0}\natural\varSigma_{g,1}^{s}$:
there is a self-embedding $e:\varSigma_{g,1}^{s}\hookrightarrow\varSigma_{g,1}^{s}$
(the complement of whose image is a disc). Then $\widehat{\omega}_{s}\circ\varrho_{s}\left(\varphi\right)$
is defined on the image of $e$ by $e\circ\varphi\circ e^{-1}$ and
by the identity on the complement of the image of $e$. Therefore
the choice of an isotopy of self-embeddings from $e$ to the identity
of $\varSigma_{g,1}^{s}$ induces a homotopy from $\widehat{\omega}_{s}\circ\varrho_{s}$
to the identity of $\textrm{Diff}^{\textrm{\ensuremath{\partial_{0}},points}}\left(\varSigma_{g,1}^{s}\right)$.
Hence, $\widehat{\omega}_{s}\circ\varrho_{s}\left(\varphi\right)$
is isotopic to $\varphi$: the composition $\widehat{\omega}_{s}\circ\varrho_{s}$
is thus isotopic to the identity. We deduce that $\pi_{0}\left(\varrho_{s}\right):\mathbf{\Gamma}_{g,1}^{\left[s\right]}\rightarrow\mathbf{\Gamma}_{g,1}^{\left[s+1\right]}$
is a $1$-sided inverse of the map $\omega_{s}:\mathbf{\Gamma}_{g,1}^{\left[s+1\right]}\rightarrow\mathbf{\Gamma}_{g,1}^{\left[s\right]}$.
Therefore the morphism $\pi_{0}\left(\varrho_{s}\right)$ is injective
and provides a splitting of the exact sequence (\ref{eq:Birman1}),
which defines an isomorphism $\mathbf{\Gamma}_{g,1}^{\left[s+1\right]}\cong\pi_{1}\left(\varSigma_{g,1}^{s},x\right)\underset{a_{\varSigma_{g,1}^{s}}^{x}}{\rtimes}\mathbf{\Gamma}_{g,1}^{\left[s\right]}$.
\end{proof}
Hence, applying Corollary \ref{cor:bigcor} to this situation, we
obtain:
\begin{prop}
\label{cor:Firststepmcg}Let $n$, $s$ and $q\geq1$ be natural numbers.
Let $M_{n}$ be a $R\left[\mathbf{\Gamma}_{n,1}^{\left[s+1\right]}\right]$-module
on which $\pi_{1}\left(\varSigma_{n,1}^{s},x\right)$ acts trivially.
Then:
\begin{equation}
H_{q}\left(\mathbf{\Gamma}_{n,1}^{\left[s+1\right]},M_{n}\right)\cong H_{q-1}\left(\mathbf{\Gamma}_{n,1}^{\left[s\right]},H_{1}\left(\varSigma_{n,1}^{s},R\right)\underset{R}{\otimes}M_{n}\right)\oplus H_{q}\left(\mathbf{\Gamma}_{n,1}^{\left[s\right]},M_{n}\right).\label{eq:firstmcgorient}
\end{equation}
\end{prop}

\paragraph{Computation of $H_{d}\left(\mathbf{\Gamma}_{\infty,1},H_{1}\left(\varSigma_{\infty,1},\mathbb{Z}\right)^{\otimes m}\right)$.}

An application of Proposition \ref{cor:Firststepmcg} is to compute
the stable homology groups $H_{d}\left(\mathbf{\Gamma}_{\infty,1},H_{1}\left(\varSigma_{\infty,1},\mathbb{Z}\right)^{\otimes m}\right)$
for all natural numbers $m$ and $d$. First, let us introduce a suitable
groupoid for our work, inspired by \cite[Section 5.6]{WahlRandal-Williams}.
We fix a unit disc with one marked point denoted by $\varSigma_{0,1}^{1}$
and a torus with one boundary component denoted by $\varSigma_{1,1}^{0}$.
Recall that by the classification of surfaces, for all $g,s\in\mathbb{N}$,
there is an homeomorphism $\varSigma_{g,1}^{s}\simeq\left(\underset{s}{\natural}\varSigma_{0,1}^{1}\right)\natural\left(\underset{g}{\natural}\varSigma_{1,1}^{0}\right)$.
\begin{defn}
\label{def:groupoidsurface}Let $\mathfrak{\mathfrak{M}}_{2}$ be
the skeleton of the groupoid defined by:

\begin{itemize}
\item Objects: the surfaces $\varSigma_{g,1}^{s}$ for all natural numbers
$g$ and $s$, with $I:\left[-1,1\right]\rightarrow\partial\varSigma_{g,1}^{s}$
a parametrized interval in the boundary and $p=0\in I$ a basepoint;
\item Morphisms: $Aut_{\mathfrak{M}_{2}}\left(\varSigma_{g,1}^{s}\right)=\mathbf{\Gamma}_{g,1}^{s}$
for all natural numbers $g$ and $s$.
\end{itemize}
\end{defn}

Let $\mathfrak{M}_{2}^{g}$ be the full subgroupoid of $\mathfrak{M}_{2}$
on the objects $\left\{ \varSigma_{n,1}\right\} _{n\in\mathbb{N}}$.
As stated in the proof of \cite[Proposition 5.18]{WahlRandal-Williams},
the boundary connected sum $\natural$ induces a strict braided monoidal
structure $\left(\mathfrak{M}_{2}^{g},\natural,\left(\varSigma_{0,1},I\right)\right)$.

For all natural numbers $g$, we denote by $a_{\varSigma_{g,1}}$
the action of the mapping class group $\mathbf{\Gamma}_{g,1}$ on
the fundamental group $\pi_{1}\left(\varSigma_{g,1},p\right)$. We
define a functor $\mathcal{A}_{\mathfrak{M}_{2}^{g}}:\mathfrak{M}_{2}^{g}\rightarrow\mathfrak{Gr}$
to be the fundamental groups $\pi_{1}\left(\varSigma_{1,1},p\right)$
and $\pi_{1}\left(\varSigma_{0,1},p\right)$ on the objects $\varSigma_{1,1}$
and $\varSigma_{0,1}$, and then inductively $\mathcal{A}_{\mathfrak{M}_{2}^{g}}\left(\varSigma_{n,1}\natural\varSigma_{n',1}\right)=\pi_{1}\left(\varSigma_{n,1},p\right)*\pi_{1}\left(\varSigma_{n',1},p\right)$
for all natural numbers $n$, and assigning the morphism $a_{\varSigma_{g,1}}\left(\varphi\right)$
for all $\varphi\in\mathbf{\Gamma}_{g,1}$. Recall that the group
$\pi_{1}\left(\varSigma_{n,1},p\right)$ is free of rank $2n$. By
Van Kampen's theorem (see for example \cite[Section 1.2]{hatcher2002algebraic}),
the group $\mathcal{A}_{\mathfrak{M}_{2}^{g}}\left(\varSigma_{n,1}\natural\varSigma_{n',1}\right)$
is isomorphic to the fundamental group of the surface $\varSigma_{n,1}\natural\varSigma_{n',1}$:
our assignment is thus consistent.

Hence, it follows from Lemma \ref{lem:ExtensionfunctorUQ}:
\begin{prop}
\label{prop:definesymplecticfunctor}The functor $\mathcal{A}_{\mathfrak{M}_{2}^{g}}:\left(\mathfrak{M}_{2}^{g},\natural,\varSigma_{0,1}^{0}\right)\rightarrow\left(\mathfrak{gr},*,0_{\mathfrak{Gr}}\right)$
is strict monoidal and extends to a functor $\pi_{1}\left(-,p\right):\mathfrak{U}\mathfrak{M}_{2}^{g}\rightarrow\mathfrak{gr}$
by assigning for all natural numbers $n$ and $n'$:
\[
\pi_{1}\left(-,p\right)\left(\left[\varSigma_{n',1},id_{\varSigma_{n'+n,1}}\right]\right)=\iota_{\pi_{1}\left(\varSigma_{n',1},p\right)}*id_{\pi_{1}\left(\varSigma_{n,1},p\right)}.
\]
\end{prop}

For all natural numbers $n$, since the free group $\pi_{1}\left(\varSigma_{n,1},x\right)$
acts trivially on the homology group $H_{1}\left(\varSigma_{n,1},\mathbb{Z}\right)$,
we have an isomorphism:
\[
H_{1}\left(\pi_{1}\left(\varSigma_{n,1},x\right),H_{1}\left(\varSigma_{n,1},\mathbb{Z}\right)^{\otimes m}\right)\cong H_{1}\left(\varSigma_{n,1},\mathbb{Z}\right)^{\otimes\left(m+1\right)}.
\]
Also, the action of $\mathbf{\Gamma}_{n,2}$ on $H_{1}\left(\varSigma_{n,1},\mathbb{Z}\right)^{\otimes m}$
is induced by the one of $\mathbf{\Gamma}_{n,1}$ via the surjections
$\omega_{0}\circ\rho:\mathbf{\Gamma}_{n,2}\twoheadrightarrow\mathbf{\Gamma}_{n,1}^{1}\twoheadrightarrow\mathbf{\Gamma}_{n,1}$.
It follows from Lemma \ref{lem:H1tensorm} that the functor $H_{1}\left(\mathcal{A}_{\mathfrak{M}_{2}^{g}},\mathbb{Z}\right)^{\otimes m}$
is very strong polynomial of degree $m$. Using the terminology of
\cite{boldsen} and \cite{cohenmadsen}, $H_{1}\left(\mathcal{A}_{\mathfrak{M}_{2}^{g}},\mathbb{Z}\right)^{\otimes m}$
is thus a coefficient system of degree $m$. Hence, it follows from
the stability results of Boldsen \cite{boldsen} or Cohen and Madsen
\cite{cohenmadsen} that:
\begin{thm}
\cite[Theorem 4.17]{boldsen}\cite[Theorem 0.4]{cohenmadsen} \label{thm:stabilitykey}Let
$m$, $n$ and $q$ be natural numbers such that $2n\geq3q+m$:
\[
H_{q}\left(\mathbf{\Gamma}_{n,2},H_{1}\left(\varSigma_{n,1},\mathbb{Z}\right)^{\otimes m}\right)\cong H_{q}\left(\mathbf{\Gamma}_{n,1},H_{1}\left(\varSigma_{n,1},\mathbb{Z}\right)^{\otimes m}\right).
\]
\end{thm}

Then, we prove:
\begin{thm}
\label{thm:homologymcgtensorproduct} Let $m$, $n$ and $q$ be natural
numbers such that $2n\geq3q+m$. Then, there is an isomorphism:
\[
H_{q}\left(\mathbf{\Gamma}_{n,1},H_{1}\left(\varSigma_{n,1},\mathbb{Z}\right)^{\otimes m}\right)\cong\underset{\left\lfloor \frac{q-1}{2}\right\rfloor \geq k\geq0}{\bigoplus}H_{q-(2k+1)}\left(\mathbf{\Gamma}_{n,1},H_{1}\left(\varSigma_{n,1},\mathbb{Z}\right)^{\otimes m-1}\right).
\]
\end{thm}

\begin{proof}
For the purposes of this proof, we abbreviate $H_{1}\left(\varSigma_{n,1},\mathbb{Z}\right)^{\otimes m}$
by $H^{\left(m\right)}$. The Lyndon-Hochschild-Serre spectral sequence
with coefficients given by $H^{\left(m\right)}$ associated with the
short exact sequence (\ref{eq:Birman2}) has only two non-trivial
rows. Hence, for all natural numbers $n\geq1$, we obtain the following
long exact sequence where we denote $\lambda_{q}=H_{q}\left(\rho,H^{\left(m\right)}\right)$:

\begin{equation}
\xymatrix{\cdots\ar@{->}[r]^{d_{q+1,0}^{2}\,\,\,\,\,\,\,\,\,\,\,\,\,\,\,\,\,\,\,\,\,\,\,\,\,\,} & H_{q-1}\left(\mathbf{\Gamma}_{n,1}^{1},H^{\left(m\right)}\right)\ar@{->}[r] & H_{q}\left(\mathbf{\Gamma}_{n,2},H^{\left(m\right)}\right)\ar@{->}[r]^{\lambda_{q}} & H_{q}\left(\mathbf{\Gamma}_{n,1}^{1},H^{\left(m\right)}\right)\ar@{->}[r]^{\,\,\,\,\,\,\,\,\,\,\,\,\,\,\,\,\,\,\,\,\,\,d_{q,0}^{2}} & \cdots}
.\label{eq:lesdeuxi=0000E8mess}
\end{equation}
Recall from Corollary \ref{cor:bigcor} that, denoting $\varphi_{q}^{sp}$
a splitting of $\varphi_{q}:H_{q-1}\left(\mathbf{\Gamma}_{n,1},H^{\left(m+1\right)}\right)\rightarrow H_{q}\left(\mathbf{\Gamma}_{n,1}^{1},H^{\left(m\right)}\right)$
(which exists by the splitting lemma for abelian groups), the isomorphism
of Proposition \ref{cor:Firststepmcg} is defined by
\[
\varphi_{q}^{sp}\oplus H_{q}\left(\omega_{0},H^{\left(m\right)}\right):H_{q}\left(\mathbf{\Gamma}_{n,1}^{1},H^{\left(m\right)}\right)\cong H_{q-1}\left(\mathbf{\Gamma}_{n,1},H^{\left(m+1\right)}\right)\oplus H_{q}\left(\mathbf{\Gamma}_{n,1},H^{\left(m\right)}\right).
\]
Let us consider the projection $pr:H_{q-1}\left(\mathbf{\Gamma}_{n,1},H^{\left(m+1\right)}\right)\oplus H_{q}\left(\mathbf{\Gamma}_{n,1},H^{\left(m\right)}\right)\twoheadrightarrow H_{q}\left(\mathbf{\Gamma}_{n,1},H^{\left(m\right)}\right)$.
Then, fixing a natural number $n$ so that $2n\geq3q+m$ and applying
Theorem \ref{thm:stabilitykey} (the isomorphism it gives being induced
by $\omega_{0}\circ\rho$), the composition
\[
H_{q}\left(\omega_{0}\circ\rho,H^{\left(m\right)}\right)^{-1}\circ pr\circ\left(\varphi_{q}^{sp}\oplus H_{q}\left(\omega_{0},H^{\left(m\right)}\right)\right)
\]
defines a splitting of $\lambda_{q}$. Therefore $\lambda_{q}$ is
split-injective and we obtain from the long exact sequence (\ref{eq:lesdeuxi=0000E8mess})
the following isomorphism for $2n\geq3q+m$:
\[
H_{q}\left(\mathbf{\Gamma}_{n,1}^{1},H^{\left(m\right)}\right)\cong H_{q}\left(\mathbf{\Gamma}_{n,2},H^{\left(m\right)}\right)\oplus H_{q-2}\left(\mathbf{\Gamma}_{n,1}^{1},H^{\left(m\right)}\right).
\]

Now we again apply Proposition \ref{cor:Firststepmcg} to the homology
groups $H_{q}\left(\mathbf{\Gamma}_{n,1}^{1},H^{\left(m\right)}\right)$
and $H_{q-2}\left(\mathbf{\Gamma}_{n,1}^{1},H^{\left(m\right)}\right)$.
Then it follows from Theorem \ref{thm:stabilitykey} that we have
the following isomorphism for $2n\geq3q+m$:

\[
H_{q-1}\left(\mathbf{\Gamma}_{n,1},H^{\left(m+1\right)}\right)\oplus H_{q}\left(\mathbf{\Gamma}_{n,2},H^{\left(m\right)}\right)\cong H_{q}\left(\mathbf{\Gamma}_{n,2},H^{\left(m\right)}\right)\oplus H_{q-2}\left(\mathbf{\Gamma}_{n,1},H^{\left(m\right)}\right)\oplus H_{q-3}\left(\mathbf{\Gamma}_{n,1},H^{\left(m+1\right)}\right).
\]
Note that the summand $H_{q}\left(\mathbf{\Gamma}_{n,2},H^{\left(m\right)}\right)$
on both side of this isomorphism is the image of the split-injective
morphism $\lambda_{q}$. Hence the image through the differential
$d_{q,0}^{2}$ gives the following isomorphism for $2n\geq3q+m$:
\[
H_{q-1}\left(\mathbf{\Gamma}_{n,1},H^{\left(m+1\right)}\right)\cong H_{q-2}\left(\mathbf{\Gamma}_{n,1},H^{\left(m\right)}\right)\oplus H_{q-3}\left(\mathbf{\Gamma}_{n,1},H^{\left(m+1\right)}\right).
\]
The result then follows by induction on $q$.
\end{proof}
\begin{rem}
In \cite[Theorem 1.B.]{Kawa1}, Kawazumi leads the analogous computation
for cohomology. The method and techniques used in \cite{Kawa1} are
different from the ones presented here. Using the Universal Coefficient-type
theorem for twisted coefficients (see for example \cite[Théorème I.5.5.2]{Godement}),
Theorem \ref{thm:homologymcgtensorproduct} recovers the computation
\cite[Theorem 1.B.]{Kawa1}.
\end{rem}

\paragraph{Computation of $H_{d}\left(\mathbf{\Gamma}_{\infty,1}^{1},\mathbb{\mathbb{Z}}\right)$.}

Another application of Proposition \ref{cor:Firststepmcg} is to compute
the stable homology groups $H_{d}\left(\mathbf{\Gamma}_{\infty,1}^{1},\mathbb{\mathbb{Z}}\right)$
for all natural numbers $d$. Using Proposition \ref{cor:Firststepmcg}
with constant module $\mathbb{Z}$ and Theorem \ref{thm:homologymcgtensorproduct}
with $m=1$, we prove:
\begin{cor}
Let $n$ and $q$ be natural numbers such that $2n\geq3q$. Then,
there is an isomorphism for all $m\geq0$:
\begin{eqnarray*}
H_{q}\left(\mathbf{\Gamma}_{n,1}^{1},H_{1}\left(\varSigma_{n,1},\mathbb{Z}\right)^{\otimes m}\right) & \cong & \underset{\left\lfloor \frac{q-1}{2}\right\rfloor \geq k\geq0}{\bigoplus}H_{q-(2k+1)}\left(\mathbf{\Gamma}_{n,1},H_{1}\left(\varSigma_{n,1},\mathbb{Z}\right)^{\otimes m-1}\right)\\
 &  & \underset{\left\lfloor \frac{q-2}{2}\right\rfloor \geq k\geq0}{\bigoplus}H_{q-(2k+2)}\left(\mathbf{\Gamma}_{n,1},H_{1}\left(\varSigma_{n,1},\mathbb{Z}\right)^{\otimes m}\right).
\end{eqnarray*}
In particular
\[
H_{q}\left(\mathbf{\Gamma}_{n,1}^{1},\mathbb{Z}\right)\cong\underset{\left\lfloor \frac{q}{2}\right\rfloor \geq k\geq0}{\bigoplus}H_{q-2k}\left(\mathbf{\Gamma}_{n,1},\mathbb{Z}\right).
\]
\end{cor}

Using other techniques (namely an equivalence of classifying spaces),
Bödigheimer and Tillmann prove the more general result:
\begin{thm}
\label{thm:resultBodiTil}\cite[Corollary 1.2]{boedigtilman2} Let
$q$ and $n$ be natural numbers such that $n\geq2q$. For all natural
numbers $s$
\[
H_{q}\left(\mathbf{\Gamma}_{n,1}^{\left[s\right]},\mathbb{Z}\right)\cong\underset{k+l=q}{\bigoplus}\left(H_{k}\left(\mathbf{\Gamma}_{n,1},\mathbb{Z}\right)\underset{\mathbb{Z}}{\otimes}H_{l}\left(\left(\mathbb{C}P^{\infty}\right)^{\times s},\mathbb{Z}\right)\right)
\]
where $\mathbb{C}P^{\infty}$ denotes the infinite dimensional complex
projective space.
\end{thm}

\subsubsection{Automorphisms of free groups with boundaries\label{subsec:autFnwithboundaries}}

Let $\mathscr{G}_{n,k}$ denote the topological space consisting of
a wedge of $n\in\mathbb{N}$ circles together with $k$ distinguished
circles joined by arcs to the basepoint. For $s\in\mathbb{N}$, let
$\mathscr{G}_{n,k}^{s}$ be the space obtained from $\mathscr{G}_{n,k}$
by wedging $s-1$ edges at the basepoint. We denote by $\textrm{Htpy}_{*}\left(\mathscr{G}_{n,k}^{s};\partial\right)$
the space of homotopy equivalences of $\mathscr{G}_{n,k}$ that preserve
the basepoint and restrict to the identity on each of the $k$ distinguished
circles and the s basepoints. Let $A_{n,k}^{s}$ be the group of path-components
of $\textrm{Htpy}_{*}\left(\mathscr{G}_{n,k}^{s};\partial\right)$.
For instance, for $n$ a natural number and denoting by $\mathbf{F}_{n}$
the free group of rank $n$, then $A_{n,0}^{1}$ is isomorphic to
the automorphism group of $\mathbf{F}_{n}$ (denoted by $Aut\left(\mathbf{F}_{n}\right)$)
and $A_{n,0}^{2}$ is isomorphic to the holomorphs of the free group
$\mathbf{F}_{n}$. We refer the reader to \cite{HatcherWahlAutFnboundaries}
and \cite{jensenhol} for more details on these groups.
\begin{rem}
\label{rem:autfreewithboundquotient}We denote by $M_{n,k}^{s}$ the
connected sum $\left(\sharp_{n}\left(\mathbb{S}^{1}\times\mathbb{S}^{2}\right)\right)\sharp\left(\sharp_{k}\left(\mathbb{S}^{1}\times\mathbb{D}^{2}\right)\right)\sharp\left(\sharp_{s}\left(\mathbb{D}^{3}\right)\right)$
and by $\pi_{0}\textrm{Diff}\left(M_{n,k}^{s}\right)$ the isotopy
classes of diffeomorphisms of $M_{n,k}^{s}$ (a.k.a. the mapping class
group of $M_{n,k}^{s}$). Recall from \cite[Theorem 1.1]{HatcherWahlAutFnboundaries}
that there is a short exact sequence
\[
\xymatrix{1\ar@{->}[r] & K_{n,k}^{s}\ar@{->}[r] & \pi_{0}\textrm{Diff}\left(M_{n,k}^{s}\right)\ar@{->}[r] & A_{n,k}^{s}\ar@{->}[r] & 1}
,
\]
where $K_{n,k}^{s}$ is generated by Dehn twists along embedded 2
spheres (and is a product of at most $n+k+s$ copies of $\mathbb{Z}/2\mathbb{Z}$).
Hence the automorphism of free groups with boundaries $A_{n,k}^{s}$
can be viewed as a quotient of the mapping class group of the $3$-manifold
$M_{n,k}^{s}$.
\end{rem}

For $k,n\in\mathbb{N}$, we denote by $Aut_{n,k}$ the subgroup of
$Aut\left(\mathbf{F}_{n+k}\right)$ of automorphisms that take each
of the last $k$ generators to a conjugate of itself. Denoting by
$\textrm{Htpy}_{*}\left(\mathscr{G}_{n,k};\left[\partial\right]\right)$
the space of homotopy equivalences of $\mathscr{G}_{n,k}$ that preserve
the basepoint and fixing the $k$ distinguished circles up to a rotation,
then $Aut_{n,k}$ is the group of path-components of $\textrm{Htpy}_{*}\left(\mathscr{G}_{n,k};\left[\partial\right]\right)$.
Restricting the elements of $\textrm{Htpy}_{*}\left(\mathscr{G}_{n,k};\left[\partial\right]\right)$
to their rotations of the $k$ distinguished circles defines a fibration
$\textrm{Htpy}_{*}\left(\mathscr{G}_{n,k};\partial\right)\rightarrow\textrm{Htpy}_{*}\left(\mathscr{G}_{n,k};\left[\partial\right]\right)\rightarrow\left(\mathbb{S}^{1}\right)^{k}$.
The homotopy long exact sequence associated with this fibration provides
an exact sequence

\[
\xymatrix{\mathbb{Z}^{k}\ar@{->}[r] & A_{n,k}\ar@{->}[r] & Aut_{n,k}\ar@{->}[r] & 1}
\]
(actually, by \cite[ Section 2]{JensenWahl}, the left hand map is
injective and this is a short exact sequence), which gives a surjective
map $A_{n,k}\twoheadrightarrow Aut_{n,k}$. For $k,n\in\mathbb{N}$,
we denote by $a_{A_{n,k}}$ the composition $A_{n,k}^{1}\rightarrow A_{n,k}\twoheadrightarrow Aut_{n,k}\hookrightarrow Aut\left(\mathbf{F}_{n+k}\right)$
where the map $A_{n,k}^{1}\rightarrow A_{n,k}$ forgets the basepoint.
We recall the following useful result:
\begin{lem}
\cite{HatcherWahlAutFnboundaries}\label{lem:splittingholomorphs}
Let $n$, $k$ and $s\geq2$ be natural numbers. There is a split
short exact sequence
\begin{equation}
\xymatrix{1\ar@{->}[r] & \mathbf{F}_{n+k}\ar@{->}[r] & A_{n,k}^{s}\ar@{->}[r] & A_{n,k}^{s-1}\ar@{->}[r] & 1,}
\label{eq:splitseshol}
\end{equation}
where the map $A_{n,k}^{s}\rightarrow A_{n,k}^{s-1}$ forgets the
last basepoint and a fortiori $A_{n,k}^{s}\cong\left(\mathbf{F}_{n+k}\right)^{s-1}\rtimes A_{n,k}^{1}$
where $A_{n,k}^{1}$ acts diagonally on $\left(\mathbf{F}_{n+k}\right)^{s-1}$
via the map $a_{A_{n,k}}:A_{n,k}^{1}\rightarrow Aut\left(\mathbf{F}_{n+k}\right)$.
\end{lem}

Let $k$ and $s$ be fixed natural numbers. Let $\mathfrak{A}_{s,k}$
be the groupoid with the spaces $\mathscr{G}_{n,k}^{s}$ as its objects
and $A_{n,k}^{s}$ as automorphism groups for all natural numbers.
For $s=1$ and $k=0$, $\mathfrak{A}_{1,0}$ is the maximal subgroupoid
of the category $\mathfrak{gr}$ of finitely generated free groups
and we denote by $i:\mathfrak{A}_{1,0}\rightarrow\mathbf{\mathfrak{gr}}$
the inclusion functor. Moreover the coproduct $*$ induces a strict
symmetric monoidal structure $\left(\mathfrak{A}_{1,0},*,0_{\mathfrak{Gr}}\right)$
by restriction and we thus define the symmetric monoidal category
$\left(\mathfrak{U}\mathfrak{A}_{1,0},*,0_{\mathfrak{Gr}}\right)$.
The functor $i$ lifts to the category $\mathfrak{U}\mathfrak{A}_{1,0}$
by sending a morphism $\left[\mathbf{F}_{n_{2}-n_{1}},g\right]:\mathbf{F}_{n_{1}}\rightarrow\mathbf{F}_{n_{2}}$
of $\mathfrak{U}\mathfrak{A}_{1,0}$ (where $g\in Aut\left(\mathbf{F}_{n_{2}}\right)$)
to the morphism $g\circ\left(\iota_{\mathbf{F}_{n_{2}-n_{1}}}*id_{\mathbf{F}_{n_{1}}}\right):\mathbf{F}_{n_{1}}\hookrightarrow\mathbf{F}_{n_{2}}$
of $\mathfrak{gr}$: since $f\circ\iota_{\mathbf{F}_{n}}=\iota_{\mathbf{F}_{n}}$
for all $f\in Aut\left(\mathbf{F}_{n}\right)$, the relation (\ref{eq:criterion'})
is trivially satisfied and the result follows from Lemma \ref{lem:criterionfamilymorphismsfunctor}.
Then $i$ is a faithful functor and $\mathfrak{U}\mathfrak{A}_{1,0}$
can be seen as a subcategory of $\mathbf{\mathfrak{gr}}$.

Let $k$ and $s\geq1$ be natural numbers. Precomposing by the surjection
$A_{n,k}^{s}\twoheadrightarrow A_{n,k}^{s-1}\twoheadrightarrow\cdots\twoheadrightarrow A_{n,k}^{1}$,
the morphisms $\left\{ a_{A_{n,k}}\right\} _{n\in\mathbb{N}}$ assemble
to define a functor $\mathcal{A}_{\mathfrak{A}_{s,k}}:\mathfrak{A}_{s,k}\rightarrow\mathfrak{Gr}$
such that $\mathcal{A}_{\mathfrak{A}_{s,k}}\left(n\right)=\mathbf{F}_{n+k}$
for all natural numbers $n$. Furthermore, we recall the stable homology
result for automorphism groups of free groups due to Galatius for
constant coefficients and Djament and Vespa for twisted coefficients:
\begin{thm}
\label{thm:resultsstablehomologyAut(Fn)} Let $q\geq1$ be a natural
number. Then:

\begin{itemize}
\item \cite{Galatius} for $n\geq2q+1$, $H_{q}\left(Aut\left(\mathbf{F}_{n}\right),\mathbb{Q}\right)=0$;
\item \cite[Théorème 1]{DV2} for $F:\mathfrak{gr}\rightarrow\mathbf{Ab}$
a strong polynomial functor such that $F\left(0_{\mathfrak{Gr}}\right)=0_{\mathfrak{Gr}}$,
then $\underset{n\in\mathbb{N}}{\textrm{Colim}}\left(H_{q}\left(Aut\left(\mathbf{F}_{n}\right),F\left(n\right)\right)\right)=0$.
\end{itemize}
\end{thm}

Hence, we can establish the main result of Section \ref{subsec:autFnwithboundaries}.
\begin{thm}
\label{thm:resultAutFnwithbound}Let $s\geq2$ and $q\geq1$ be natural
numbers.

\begin{enumerate}
\item Let $F:\mathfrak{gr}\rightarrow\mathbf{Ab}$ be a strong polynomial
functor such that $F\left(0_{\mathfrak{Gr}}\right)=0_{\mathfrak{Gr}}$.
The action of $A_{n,0}^{s}$ on $F\left(n\right)$ is induced by the
surjections $A_{n,0}^{s}\twoheadrightarrow A_{n,0}^{s-1}\twoheadrightarrow\cdots\twoheadrightarrow A_{n,0}^{1}.$
Then
\[
\underset{n\in\mathbb{N}}{\textrm{Colim}}\left(H_{q}\left(A_{n,0}^{s},F\left(n\right)\right)\right)=0.
\]
\item For all  natural numbers $n\geq2q+2$ and $k\geq0$, $H_{q}\left(A_{n,k}^{s},\mathbb{Q}\right)=0$.
\end{enumerate}
\end{thm}

\begin{proof}
Since $i:\mathfrak{U}\mathfrak{A}_{1,0}\rightarrow\mathbf{\mathfrak{gr}}$
is faithful and the monoidal structure of $\mathfrak{U}\mathfrak{A}_{1,0}$
is induced from the one of $\mathbf{\mathfrak{gr}}$, $i$ is strong
monoidal. Recall from Remark \ref{rem:caseunitterminal} that, since
the trivial group $0_{\mathfrak{Gr}}$ is a terminal object of $\mathfrak{gr}$,
$F$ is very strong polynomial. We then deduce from Proposition \ref{prop:precompositionstrongmono}
that the functor $F\circ i:\mathfrak{U}\mathfrak{A}_{1,0}\rightarrow\mathbf{Ab}$
is very strong polynomial. It follows from Lemma \ref{cor:increasepoly}
that $H_{1}\left(\mathbf{F}_{-},\mathbb{Q}\right)\underset{\mathbb{Q}}{\otimes}F\left(-\right):\mathfrak{U}\mathfrak{A}_{1,0}\overset{i}{\rightarrow}\mathfrak{gr}\rightarrow\mathbf{Ab}$
is a strong polynomial functor. Hence, the first result follows from
Corollary \ref{cor:bigcor} and Theorem \ref{thm:resultsstablehomologyAut(Fn)}.

In \cite[Theorem 1.1]{HatcherWahl2}, Hatcher and Wahl prove that
the stabilization morphism $A_{n,k}^{s}\rightarrow A_{n,k+1}^{s}$
induces an isomorphism for the rational homology $H_{q}\left(A_{n,k}^{s},\mathbb{Q}\right)\overset{\sim}{\rightarrow}H_{q}\left(A_{n,k+1}^{s},\mathbb{Q}\right)$
if $n\geq3q+3$. The second result thus follows from the previous
statement.
\end{proof}
\begin{rem}
For $k=0$ and $s=2$, Theorem \ref{thm:resultAutFnwithbound} recovers
the results \cite[Theorem 1.2 (b) and (c)]{jensenhol} due to Jensen.
\end{rem}

\section{Twisted stable homologies for $FI$-modules\label{sec:FI-Mod}}

In this section, we present a general principle to compute the twisted
stable homology for mapping class groups with non-trivial finite quotient
groups. First, we give a general decomposition for the twisted stable
homology using functor homology Section \ref{sec:A-general-result}.
Then, we can establish in Theorem \ref{cor:Mainresultparticularkernels}
a general formula to compute the stable homology with twisted coefficients
given by functors over categories associated with the aforementioned
finite quotient groups in Section \ref{subsec:Generalpropertykernels}.
This allows one to set explicit formulas for the stable homology with
coefficients given by $FI$-modules for braid groups, mapping class
groups of orientable surfaces and some particular right-angled Artin
groups in Section \ref{subsec:Applications-1}. \textbf{Throughout
Section \ref{sec:FI-Mod}, we fix $\mathbb{K}$ a field.}

\subsection{General decomposition for the twisted stable homology using functor
homology\label{sec:A-general-result}}

In this first subsection, we prove a decomposition result for the
stable homology with twisted coefficients for families of groups whose
associated groupoid is braided strict monoidal and satisfies the assumptions
of Theorem \ref{thm:consequencecondcancalandinject} (see Theorem
\ref{thm:step1}). It extends a previous analogous result due Djament
and Vespa in \cite[Section 1 and 2]{DV1} when the ambient monoidal
structure is symmetric. It will be a key step to prove Theorem \ref{cor:Mainresultparticularkernels}.
We refer the reader to the papers \cite[Section 2]{FranjouPira} and
\cite[Appendice A]{DV1} for an introduction to homological algebra
in functor categories and we assume that all the definitions, properties
and results there are known.

\textbf{Throughout Section \ref{sec:A-general-result}, we consider
$\left(\mathcal{G},\natural,0\right)$ a small braided strict monoidal
groupoid, with objects indexed by the natural numbers.} An object
of $\mathcal{G}$ is thus denoted by $\underline{n}$, where $n$
is its corresponding indexing natural number. We denote the automorphism
group $Aut_{\mathcal{G}}\left(\underline{n}\right)$ by $G_{n}$ and
assume that there are no morphisms between distinct objects. These
groups define a family of groups $G_{-}:\left(\mathbb{N},\leq\right)\rightarrow\mathfrak{Gr}$
such that $G_{-}\left(\gamma{}_{n}\right):G_{n}\rightarrow G_{n+1}$
is the injective group morphism which sends $\varphi\in G_{n}$ to
$\varphi\natural id_{\underline{1}}$.

We assume that $\mathcal{G}$ has no zero divisors, that $Aut_{\mathcal{G}}(0)=\left\{ id_{0}\right\} $
and that it satisfies the properties $\mathbf{\left(C\right)}$ and
$\mathbf{\left(I\right)}$ of Definition \ref{def:conditionscancelandinject}.
By Theorem \ref{thm:consequencecondcancalandinject}, the monoidal
structure $\natural$ extends to Quillen's bracket construction and
defines pre-braided homogeneous category $\left(\mathfrak{U}\mathcal{G},\natural,0\right)$.
Also, the unit $0$ is an initial object in $\mathfrak{U}\mathcal{G}$
and we recall that $\iota_{\underline{n}}=\left[\underline{n},id_{\underline{n}}\right]:0\rightarrow\underline{n}$
denotes the unique morphism in $\mathfrak{U}\mathcal{G}$ from $0$
to $\underline{n}$. Hence, we have canonical morphisms $id_{\underline{n}}\natural\iota_{\underline{n'-n}}:\underline{n}\rightarrow\underline{n'}$
in $\mathfrak{U}\mathcal{G}$, for all natural numbers $n$ and $n'$
such that $n'\geq n$.

We fix $F$ an object of $\mathbf{Fct}\left(\mathfrak{U}\mathcal{G},\mathbb{K}\textrm{-}\mathfrak{Mod}\right)$.
Our goal is to compute the stable homology of the family of groups
$G_{-}$ with coefficients given by $F$. Namely, we are interested
in the computation of
\[
\underset{n\in\left(\mathbb{N},\leq\right)}{\textrm{Colim}}\left(H_{*}\left(G_{n},F\left(\underline{n}\right)\right)\right)
\]
the colimit of the homology groups $H_{*}\left(G_{n},F\left(\underline{n}\right)\right)$
with respect to the morphisms
\[
H_{*}\left(G_{-}\left(\gamma{}_{n}\right),F\left(id_{\underline{n}}\natural\iota_{\underline{1}}\right)\right):H_{*}\left(G_{n},F\left(\underline{n}\right)\right)\rightarrow H_{*}\left(G_{n+1},F\left(\underline{n+1}\right)\right)
\]
induced by the functoriality in two variables of group homology with
respect to the morphisms $G_{-}\left(\gamma{}_{n}\right)$ and $F\left(id_{\underline{n}}\natural\iota_{\underline{1}}\right)$).
\begin{notation}
We denote by $G_{\infty}$ the colimit of the groups $G_{n}$ with
respect to the morphisms $G_{-}\left(\gamma{}_{n}\right)$ and by
$G_{-}\left(\gamma{}_{\infty}\right):G_{n}\rightarrow G_{\infty}$
the associated canonical group morphism. Let $F_{\infty}$ be the
colimit of the $G_{n}$-modules $F\left(\underline{n}\right)$ with
respect to the morphisms $F\left(id_{\underline{n}}\natural\iota_{\underline{1}}\right)$.
Let $H_{*}\left(G_{\infty},F_{\infty}\right)$ be the homology group
of the colimit $G_{\infty}$ with coefficient in the colimit $F_{\infty}$:
it coincides with $\underset{n\in\left(\mathbb{N},\leq\right)}{\textrm{Colim}}\left(H_{*}\left(G\left(\underline{n}\right),F\left(\underline{n}\right)\right)\right)$
(since group homology commutes with filtered colimits).
\end{notation}

As categories with one object, the groups $\left\{ G_{n}\right\} _{n\in\mathbb{N}}$
are subcategories of $\mathfrak{U}\mathcal{G}$: for each natural
number $n$, we have a canonical faithful functor $\mathbf{i}_{G_{n}}:G_{n}\rightarrow\mathfrak{U}\mathcal{G}$.
We denote by $\Pi:G_{\infty}\times\mathfrak{U}\mathcal{G}\rightarrow\mathfrak{U}\mathcal{G}$
the projection functor and by $\Pi^{*}$ the precomposition by $\Pi$.
Thanks to the functoriality of the homology of categories, we define
for each natural number $n$ a morphism
\[
\varPsi_{F,n}:H_{*}\left(G_{n},F\left(n\right)\right)\rightarrow H_{*}\left(G_{\infty}\times\mathfrak{U}\mathcal{G},\Pi^{*}F\right)
\]
induced by restriction along the functor $\left(G_{-}\left(\gamma{}_{\infty}\right)\times\mathbf{i}_{G_{n}}\right)\circ\Delta_{G_{n}}:G_{n}\rightarrow G_{\infty}\times\mathfrak{U}\mathcal{G}$,
where $\Delta_{G_{n}}:G_{n}\rightarrow G_{n}\times G_{n}$ is the
diagonal functor. Then it formally follows from the previous definitions
that for all natural numbers $n$:
\[
H_{*}\left(G_{-}\left(\gamma{}_{n}\right),F\left(id_{\underline{n}}\natural\iota_{\underline{1}}\right)\right)\circ\varPsi_{F,n+1}=\varPsi_{F,n}.
\]
A fortiori the morphisms $\left\{ \varPsi_{F,n}\right\} _{n\in\mathbb{N}}$
are natural with respect to $\left(\mathbb{N},\leq\right)$. Hence,
the colimit with respect to $\left(\mathbb{N},\leq\right)$ defines
a unique morphism:
\[
\varPsi_{F}:H_{*}\left(G_{\infty},F_{\infty}\right)\rightarrow H_{*}\left(G_{\infty}\times\mathfrak{U}\mathcal{G},\Pi^{*}F\right).
\]
Let us state the main result of this section.
\begin{thm}
\label{thm:step1}Let $\mathbb{K}$ be a field and $\left(\mathfrak{U}\mathcal{G},\natural,0\right)$
be a pre-braided homogeneous category as detailed before. For all
functors $F:\mathfrak{U}\mathcal{G}\rightarrow\mathbb{K}\textrm{-}\mathfrak{Mod}$,
the morphism $\varPsi_{F}$ is a $\mathbb{K}$-modules isomorphism.
Moreover, $\varPsi_{F}$ decomposes as a natural isomorphism: 
\[
H_{*}\left(G_{\infty},F_{\infty}\right)\cong\underset{k+l=*}{\bigoplus}\left(H_{k}\left(G_{\infty},\mathbb{K}\right)\underset{\mathbb{K}}{\otimes}H_{l}\left(\mathfrak{U}\mathcal{G},F\right)\right).
\]
\end{thm}

\begin{proof}
Note that the morphism $\varPsi_{F}$ is a morphism of $\delta$-functors
commuting with filtered colimits. Recall that the category $\mathbf{Fct}\left(\mathfrak{U}\mathcal{G},\mathbb{K}\textrm{-}\mathfrak{Mod}\right)$
has enough projectives, provided by direct sums of the standard projective
generators functors $P_{\underline{n}}^{\mathfrak{U}\mathcal{G}}=\mathbb{K}\left[Hom_{\mathfrak{U}\mathcal{G}}\left(\underline{n},-\right)\right]$
for all natural numbers $n$. Therefore, we only have to show that
$\varPsi_{F}$ is an isomorphism when $F=P_{\underline{n}}^{\mathfrak{U}\mathcal{G}}$.
Indeed, for an ordinary functor $F$, there is an epimorphism from
a direct sum of standard projective generators to $F$ and $\mathbb{K}\textrm{-}\mathfrak{Mod}$
is an abelian category: the result then follows from a straightforward
recursion on the homological degree on the long exact sequence induced
by this epimorphism.

We deduce from Lemma \ref{rem:Homogenousgroupset} that we have the
following isomorphism of $G_{m}$-sets for all natural numbers $m\geq n$:
\[
Hom_{\mathfrak{U}\mathcal{G}}\left(\underline{n},\underline{m}\right)\cong G_{m}/G_{m-n}.
\]
Hence $P_{\underline{n}}^{\mathfrak{U}\mathcal{G}}\left(\underline{m}\right)\cong\mathbb{K}\left[G_{m}\right]\underset{\mathbb{K}\left[G_{m-n}\right]}{\otimes}\mathbb{K}$
as $G_{m}$-modules. Therefore, it follows from Shapiro's lemma that:
\[
H_{*}\left(G_{m},P_{\underline{n}}^{\mathfrak{U}\mathcal{G}}\left(\underline{m}\right)\right)\cong H_{*}\left(G_{m-n},\mathbb{K}\right).
\]
Taking the colimit with respect to $m$, we deduce the isomorphism
$H_{*}\left(G_{\infty},\left(P_{\underline{n}}^{\mathfrak{U}\mathcal{G}}\right)_{\infty}\right)\cong H_{*}\left(G_{\infty},\mathbb{K}\right)$.
Recall that $P_{\underline{n}}^{\mathfrak{U}\mathcal{G}}$ is a projective
object in $\mathfrak{U}\mathcal{G}$. Then, using the first Künneth
spectral sequence for the product of two categories, the morphism
$\varPsi_{P_{\underline{n}}^{\mathfrak{U}\mathcal{G}}}$ identifies
with:
\[
H_{*}\left(G_{\infty},\left(P_{\underline{n}}^{\mathfrak{U}\mathcal{G}}\right)_{\infty}\right)\cong H_{*}\left(G_{\infty},\mathbb{K}\right)\cong H_{*}\left(G_{\infty}\times\mathfrak{U}\mathcal{G},\Pi^{*}\left(P_{\underline{n}}^{\mathfrak{U}\mathcal{G}}\right)\right).
\]
The second part of the statement follows applying Künneth formula
for homology of categories.
\end{proof}

\subsection{Framework and first equivalence for stable homology\label{subsec:Generalpropertykernels}}

\textbf{Throughout the remainder Section \ref{sec:FI-Mod}, we assume
that the field $\mathbb{K}$ is of characteristic $0$.}

We consider three families of groups $K_{-}$, $G_{-}$ and $C_{-}$
which fit into the following short exact sequence in the category
$\mathbf{Fct}\left(\left(\mathbb{N},\leq\right),\mathfrak{Gr}\right)$:
\begin{equation}
\xymatrix{0\ar@{->}[r] & K_{-}\ar@{->}[r]^{k} & G_{-}\ar@{->}[r]^{c} & C_{-}\ar@{->}[r] & 0}
,\label{eq:1}
\end{equation}
where $k:K_{-}\rightarrow G_{-}$ and $c:G_{-}\rightarrow C_{-}$
are natural transformations and $0$ denotes the constant object of
$\mathbf{Fct}\left(\left(\mathbb{N},\leq\right),\mathfrak{Gr}\right)$
at $0_{\mathfrak{Gr}}$.

Let $\mathcal{K}$, $\mathcal{G}$ and $\mathcal{C}$ denote the groupoids
with objects indexed by natural numbers, with no morphisms between
distinct objects, such that $Aut_{\mathcal{K}}\left(\underline{n}\right)=K_{n}$,
$Aut_{\mathcal{G}}\left(\underline{n}\right)=G_{n}$ and $Aut_{\mathcal{C}}\left(\underline{n}\right)=C_{n}$.
We assume that the groupoids $\mathcal{G}$ and $\mathcal{C}$ are
endowed with braided strict monoidal structures $\left(\mathcal{G},\natural_{\mathcal{G}},0_{\mathcal{G}}\right)$
and $\left(\mathcal{C},\natural_{\mathcal{C}},0_{\mathcal{C}}\right)$,
where $\natural_{\mathcal{G}}$ and $\natural_{\mathcal{C}}$ are
defined by the addition on objects, such that:

\begin{itemize}
\item the morphisms $\left\{ c_{n}\right\} _{n\in\mathbb{N}}$ induce a
strict monoidal functor $\mathfrak{c}:\mathcal{G}\rightarrow\mathcal{C}$
defined by the identity on objects;
\item $G_{-}\left(\gamma_{n}\right)=id_{\underline{1}}\natural_{\mathcal{G}}-:G_{n}\hookrightarrow G_{n+1}$
and $C_{-}\left(\gamma_{n}\right)=id_{\underline{1}}\natural_{\mathcal{C}}-:C_{n}\hookrightarrow C_{n+1}$
for all natural numbers $n$.
\end{itemize}
Recall that the associated Quillen's bracket construction $\left(\mathfrak{U}\mathcal{G},\natural_{\mathcal{G}},0\right)$
and $\left(\mathfrak{U}\mathcal{C},\natural_{\mathcal{C}},0\right)$
are pre-braided strict monoidal by Proposition \ref{prop:Quillen'sconstructionprebraided}.
Let $\mathscr{O}_{\mathcal{G}}':\left(\mathbb{N},\leq\right)\rightarrow\mathfrak{U}\mathcal{G}$
and $\mathscr{O}_{\mathcal{C}}':\left(\mathbb{N},\leq\right)\rightarrow\mathfrak{U}\mathcal{C}$
be the faithful and essentially surjective functors assigning $\mathscr{O}_{\mathcal{G}}'\left(n\right)=\mathscr{O}_{\mathcal{C}}'\left(n\right)=\underline{n}$
and $\mathscr{O}_{\mathcal{G}}'\left(\gamma_{n}\right)=\mathscr{O}_{\mathcal{C}}'\left(\gamma_{n}\right)=\left[\underline{1},id_{\underline{n+1}}\right]$
for all natural numbers $n$. Using the functors $\mathscr{O}_{\mathcal{G}}'$
and $\mathscr{O}_{\mathcal{C}}'$, the natural transformation $c:G_{-}\rightarrow C_{-}$
identifies the morphisms $\left[\underline{n'-n},id_{\underline{n'}}\right]$
(with natural numbers $n'\geq n$) of $\mathfrak{U}\mathcal{G}$ and
$\mathfrak{U}\mathcal{C}$. The criteria (\ref{eq:criterion}) and
(\ref{eq:criterion'}) of Lemma \ref{lem:criterionfamilymorphismsfunctor}
being trivially checked, the functor $\mathfrak{c}:\mathcal{G}\rightarrow\mathcal{C}$
lifts to a functor $\mathfrak{U}\mathcal{G}\rightarrow\mathfrak{U}\mathcal{C}$,
again denoted by $\mathfrak{c}$ (abusing the notation).

The short exact sequence (\ref{eq:1}) implies that the braided strict
monoidal structure $\left(\mathcal{G},\natural_{\mathcal{G}},0_{\mathcal{G}}\right)$
induces a braided strict monoidal structure on $\mathcal{K}$, denoted
by $\left(\mathcal{K},\natural_{\mathcal{G}},0_{\mathcal{G}}\right)$,
such that: 
\[
K_{-}\left(\gamma_{n}\right)=id_{1}\natural_{\mathcal{G}}-:K_{n}\hookrightarrow K_{n+1}
\]
for all natural numbers $n$. As for the morphisms $\left\{ c_{n}\right\} _{n\in\mathbb{N}}$,
the morphisms $\left\{ k_{n}\right\} _{n\in\mathbb{N}}$ induce a
strict monoidal functor $\mathfrak{k}:\mathfrak{U}\mathcal{K}\rightarrow\mathfrak{U}\mathcal{G}$.\\

\textbf{We fix $F$ an object of $\mathbf{Fct}\left(\mathfrak{U}\mathcal{G},\mathbb{K}\textrm{-}\mathfrak{Mod}\right)$.}
For all natural numbers $n$, we abuse the notation and write $F\left(n\right)$
for the restriction of $F\left(n\right)$ from $G_{n}$ to $K_{n}$.
Our aim is to compute the stable homology $H_{*}\left(G_{\infty},F_{\infty}\right)$
of the family of groups $G_{-}$ under the assumption that $C_{-}$
is a family of finite groups. A first step is given by the following
result:
\begin{prop}
\label{prop:firstpourri} We assume that the group $C_{n}$ is finite
for each natural number $n$. Then for all natural numbers $q$:
\begin{equation}
H_{q}\left(G_{n},F\left(\underline{n}\right)\right)\cong H_{0}\left(C_{n},H_{q}\left(K_{n},F\left(\underline{n}\right)\right)\right).\label{eq:firstresultpourrihomology}
\end{equation}
\end{prop}

\begin{proof}
Applying the Lyndon-Hochschild-Serre spectral sequence for the short
exact sequence (\ref{eq:1}), we obtain the following convergent first
quadrant spectral sequence:

\begin{equation}
E_{pq}^{2}:H_{p}\left(C_{n},H_{q}\left(K_{n},F\left(\underline{n}\right)\right)\right)\Longrightarrow H_{p+q}\left(G_{n},F\left(\underline{n}\right)\right).\label{eq:2.2-1}
\end{equation}
Fixing $n$ a natural number, we have for $p\neq0$:
\[
H_{p}\left(C_{n},H_{q}\left(K_{n},F\left(\underline{n}\right)\right)\right)=0,
\]
since $C_{n}$ is a finite group. Hence, the second page of the spectral
sequence (\ref{eq:2.2-1}) has non-zero terms only on the $0$-th
column and zero differentials. A fortiori, the convergence gives that
$E^{2}=E^{\infty}$ and this gives the desired result.
\end{proof}
Let us now focus on a key property for the homologies of the kernels
$\left\{ K_{n}\right\} _{n\in\mathbb{N}}$ which improves Proposition
(\ref{prop:firstpourri}). Recall that, as $K_{n}$ is a normal subgroup
of $G_{n}$, the map $conj_{n}:G_{n}\rightarrow Aut_{\mathfrak{Gr}}\left(K_{n}\right)$
sending an element $g\in G_{n}$ to the left conjugation by $g$ is
a group morphism.
\begin{lem}
\label{lem:Extension of the functor}We define a functor $\tilde{K}_{-}:\mathfrak{U}\mathcal{G}\rightarrow\mathfrak{Gr}$
assigning $\tilde{K}_{-}\left(\underline{n}\right)=K_{n}$ for all
natural numbers $n$ and:

\begin{enumerate}
\item for all $g\in G_{n}$, $\tilde{K}_{-}\left(g\right)\in Aut_{\mathfrak{Gr}}\left(K_{n}\right)$
to be $conj_{n}\left(g\right):k\mapsto gkg^{-1}$ for all $k\in K_{n}$,
\item $\tilde{K}_{-}\left(\left[1,id_{\underline{n+1}}\right]\right)=id_{1}\natural_{\mathcal{G}}-$.
\end{enumerate}
\end{lem}

\begin{proof}
It follows from the first assignment of Lemma \ref{lem:Extension of the functor}
that we define a functor $\tilde{K}_{-}:\mathcal{G}\rightarrow\mathfrak{Gr}$.
The relation (\ref{eq:criterion}) of Lemma \ref{lem:criterionfamilymorphismsfunctor}
follows from the definition of the monoidal product $\natural_{\mathcal{G}}$.
Let $n$ and $n'$ be natural numbers such that $n'\geq n$, let $g\in G_{n}$
and $g'\in G_{n'}$. We compute for all $k\in K_{n}$:
\begin{eqnarray*}
\left(\tilde{K}_{-}\left(g'\natural_{\mathcal{G}}g\right)\circ\tilde{K}_{-}\left(\left[\underline{n'},id_{\underline{n'+n}}\right]\right)\right)\left(k\right) & = & \left(g'\natural_{\mathcal{G}}g\right)\left(id_{n'}\natural_{\mathcal{G}}k\right)\left(g'\natural_{\mathcal{G}}g\right)^{-1}\\
 & = & id_{n'}\natural_{\mathcal{G}}\left(gkg^{-1}\right)\\
 & = & \left(\tilde{K}_{-}\left(\left[\underline{n'},id_{\underline{n'+n}}\right]\right)\circ\tilde{K}_{-}\left(g\right)\right)\left(k\right).
\end{eqnarray*}
Hence, the relation (\ref{eq:criterion'}) is satisfied a fortiori
the result follows from Lemma \ref{lem:criterionfamilymorphismsfunctor}.
\end{proof}
Lemma \ref{lem:Extension of the functor} is useful to prove the following
key result. 
\begin{prop}
\label{cor:funcorhomologyug}\label{Prop:particularkernels}For all
natural numbers $q$, the homology groups $\left\{ H_{q}\left(K_{n},F\left(\underline{n}\right)\right)\right\} _{n\in\mathbb{N}}$
define a functor $H_{q}\left(K_{-},F\left(-\right)\right):\mathfrak{U}\mathcal{C}\rightarrow\mathbb{K}\textrm{-}\mathfrak{Mod}$.
\end{prop}

\begin{proof}
Let $\mathscr{P}$ be the category of pairs $\left(G,M\right)$ where
$G$ is a group and $M$ is a $G$-module for objects; for $\left(G,M\right)$
and $\left(G',M'\right)$ objects of $\mathscr{P}$, a morphism from
$\left(G,M\right)$ to $\left(G',M'\right)$ is a pair $\left(\varphi,\alpha\right)$
where $\varphi\in Hom_{\mathfrak{Gr}}\left(G,G'\right)$ and $\alpha:M\rightarrow M'$
is a $G$-module morphism, where $M'$ is endowed with a $G$-module
structure via $\varphi$. Using the functor $F:\mathfrak{U}\mathcal{G}\rightarrow\mathbb{K}\textrm{-}\mathfrak{Mod}$,
by Lemma \ref{lem:Extension of the functor} $\tilde{K}_{-}$ defines
a functor $\left(\tilde{K}_{-},F\left(-\right)\right):\mathfrak{U}\mathcal{G}\rightarrow\mathscr{P}$.
Recall from \cite[Section 8]{brown} that group homology defines a
covariant functor $H_{*}:\mathscr{P}\rightarrow\mathbb{K}\textrm{-}\mathfrak{Mod}$
for all $q\in\mathbb{N}$. Hence the composition with the functor
$\left(\tilde{K}_{-},F\left(-\right)\right):\mathfrak{U}\mathcal{G}\rightarrow\mathscr{P}$
gives a functor:
\[
H_{q}\left(K_{-},F\left(-\right)\right):\mathfrak{U}\mathcal{G}\rightarrow\mathbb{K}\textrm{-}\mathfrak{Mod}.
\]
Moreover, since inner automorphisms act trivially in homology, we
deduce that for all natural numbers $n$, the conjugation action of
$G_{n}$ on $\left(K_{n},F\left(n\right)\right)$ induces an action
of $C_{n}$ on $H_{*}\left(K_{n},F\left(n\right)\right)$. The monoidal
structures $\left(\mathcal{G},\natural_{\mathcal{G}},0_{\mathcal{G}}\right)$
and $\left(\mathcal{C},\natural_{\mathcal{C}},0_{\mathcal{C}}\right)$
being compatible, we deduce that the functor $H_{q}\left(K_{-},F\left(-\right)\right)$
factors through the category $\mathfrak{U}\mathcal{C}$ using the
functor $\mathfrak{c}:\mathfrak{U}\mathcal{G}\rightarrow\mathfrak{U}\mathcal{C}$.
\end{proof}
Finally, we recall the following classical property for the homology
of a category:
\begin{prop}
\cite[Example 2.5]{FranjouPira}\label{prop:H0categoryiscolimit}
Let $\mathfrak{C}$ be an object of $\mathfrak{Cat}$ and let $F$
be an object of $\mathbf{Fct}\left(\mathfrak{C},R\textrm{-}\mathfrak{Mod}\right)$.
Then, $H_{0}\left(\mathfrak{C},F\right)$ is isomorphic to the colimit
over $\mathfrak{C}$ of the functor $F:\mathfrak{C}\rightarrow R\textrm{-}\mathfrak{Mod}$.
\end{prop}

We thus deduce from Propositions \ref{prop:firstpourri} and \ref{cor:funcorhomologyug}:
\begin{thm}
\label{cor:Mainresultparticularkernels} Let $K_{-}$, $G_{-}$ and
$C_{-}$ three families of groups fitting in the short exact sequence
(\ref{eq:1}), such that the group $C_{n}$ is finite for all natural
numbers $n$ and the groupoids $\mathcal{G}$ and $\mathcal{C}$ are
endowed with the aforementioned braided strict monoidal structures
$\left(\mathcal{G},\natural_{\mathcal{G}},0_{\mathcal{G}}\right)$
and $\left(\mathcal{C},\natural_{\mathcal{C}},0_{\mathcal{C}}\right)$.
Then, for all natural numbers $q$:
\begin{eqnarray*}
H_{q}\left(G_{\infty},F_{\infty}\right) & \cong & \underset{l\in\mathfrak{U}\mathcal{C}}{Colim}\left(H_{q}\left(K_{l},F\left(\underline{l}\right)\right)\right).
\end{eqnarray*}
Moreover, if $F$ factors through the category $\mathfrak{U}\mathcal{C}$
(in other words, $F:\mathfrak{U}\mathcal{G}\overset{\mathfrak{c}}{\rightarrow}\mathfrak{U}\mathcal{C}\rightarrow\mathbb{K}\textrm{-}\mathfrak{Mod}$),
then:
\[
H_{q}\left(G_{\infty},F_{\infty}\right)\cong\underset{l\in\mathfrak{U}\mathcal{C}}{Colim}\left(H_{q}\left(K_{l},\mathbb{K}\right)\underset{\mathbb{K}}{\otimes}F\left(\underline{l}\right)\right).
\]
\end{thm}

\begin{proof}
Applying Theorem \ref{thm:step1} to Proposition \ref{prop:firstpourri},
we obtain:
\[
\underset{n\in\mathbb{N}}{Colim}\left(H_{0}\left(C_{n},H_{q}\left(K_{n},F\left(\underline{n}\right)\right)\right)\right)\cong\underset{n\in\mathbb{N}}{Colim}\left(H_{0}\left(C_{n},\mathbb{K}\right)\underset{\mathbb{K}}{\otimes}H_{0}\left(\mathfrak{U}\mathcal{C},H_{q}\left(K_{-},F\right)\right)\right).
\]
By Proposition \ref{prop:H0categoryiscolimit}, $H_{0}\left(\mathfrak{U}\mathcal{C},H_{q}\left(K_{-},F\right)\right)\cong\underset{l\in\mathfrak{U}\mathcal{C}}{Colim}\left(H_{q}\left(K_{l},F\left(\underline{l}\right)\right)\right)$.
Since $H_{0}\left(C_{n},\mathbb{K}\right)\cong\mathbb{K}$, we deduce
the first result. For the second result, requiring $F$ to factor
through $\mathfrak{U}\mathcal{C}$ is actually necessary to define
the functor $H_{q}\left(K_{-},\mathbb{K}\right)\underset{\mathbb{K}}{\otimes}F\left(-\right):\mathfrak{U}\mathcal{C}\rightarrow\mathbb{K}\textrm{-}\mathfrak{Mod}$
using the pointwise tensor product of functors. Then the isomorphism
follows from the universal coefficient theorem for homology.
\end{proof}
\begin{rem}
We can generalise Section \ref{subsec:Generalpropertykernels} to
the setting where $\mathbb{K}$ is a field of positive characteristic
coprime to the cardinality of $C_{n}$ for all $n$: the key point
is that the homology group $H_{p}\left(C_{n},H_{q}\left(K_{n},F\left(\underline{n}\right)\right)\right)=0$
has to vanish for $p\neq0$ which is still true under this alternative
assumption (since the multiplication by $\left|C_{n}\right|$ defines
an isomorphism of $H_{q}\left(K_{n},F\left(\underline{n}\right)\right)$).
\end{rem}

\subsection{Applications\label{subsec:Applications-1}}

We present now how to apply the general result of Theorem \ref{cor:Mainresultparticularkernels}
for various families of groups. Beforehand, we fix some notations.
We denote by $\mathfrak{S}_{n}$ the symmetric group on $n$ elements
and by $\mathfrak{S}_{-}:\left(\mathbb{N},\leq\right)\rightarrow\mathfrak{Gr}$
the family of groups defined by $\mathfrak{S}_{-}\left(n\right)=\mathfrak{S}_{n}$
and $\mathfrak{S}_{-}\left(\gamma{}_{n}\right)=id_{1}\sqcup-$ for
all natural numbers $n$.

Let $\Sigma$ be the skeleton of the groupoid of finite sets and bijections.
Note that $Obj\left(\Sigma\right)\cong\mathbb{N}$ and that the automorphism
groups are the symmetric groups $\mathfrak{S}_{n}$. The disjoint
union of finite sets $\sqcup$ induces a monoidal structure $\left(\Sigma,\sqcup,0\right)$,
the unit $0$ being the empty set. This groupoid is symmetric monoidal,
the symmetry being given by the canonical bijection $n_{1}\sqcup n_{2}\overset{\sim}{\rightarrow}n_{2}\sqcup n_{1}$
for all natural numbers $n_{1}$ and $n_{2}$. The category $\mathfrak{U}\Sigma$
is equivalent to the category of finite sets and injections $FI$
studied in \cite{ChurchEllenbergFarbstabilityFI}.

\subsubsection{Braid groups}

We respectively denote by $\mathbf{B}_{n}$ the braid group on $n$
strands and by $\mathbf{PB}_{n}$ the pure braid group on $n$ strands.
The braid groupoid $\boldsymbol{\beta}$ is the groupoid with objects
the natural numbers $n\in\mathbb{N}$ and braid groups as automorphism
groups. It is endowed with a strict braided monoidal product $\natural:\boldsymbol{\beta}\times\boldsymbol{\beta}\longrightarrow\boldsymbol{\beta}$,
defined by the usual addition for the objects and laying two braids
side by side for the morphisms. The object $0$ is the unit of this
monoidal product. The braiding of the strict monoidal groupoid $\left(\boldsymbol{\beta},\natural,0\right)$
is defined for all natural numbers $n$ and $m$ by:
\[
b_{n,m}^{\boldsymbol{\beta}}=\left(\sigma_{m}\circ\cdots\circ\sigma_{2}\circ\sigma_{1}\right)\circ\cdots\circ\left(\sigma_{n+m-2}\circ\cdots\circ\sigma_{n}\circ\sigma_{n-1}\right)\circ\left(\sigma_{n+m-1}\circ\cdots\circ\sigma_{n+1}\circ\sigma_{n}\right)
\]
where $\left\{ \sigma_{i}\right\} _{i\in\left\{ 1,\ldots,n+m-1\right\} }$
denote the Artin generators of the braid group $\mathbf{B}_{n+m}$.
We refer the reader to \cite[Chapter XI, Section 4]{MacLane1} for
more details. By \cite[Proposition 5.18]{WahlRandal-Williams}, the
category $\mathfrak{U}\boldsymbol{\beta}$ is pre-braided homogeneous\textit{.}

The classical surjections $\left\{ \mathbf{B}_{n}\overset{\mathfrak{p}_{n}}{\twoheadrightarrow}\mathfrak{S}_{n}\right\} _{n\in\mathbb{N}}$,
sending each Artin generator $\sigma_{i}\in\mathbf{B}_{n}$ to the
transposition $\tau_{i}\in\mathfrak{S}_{n}$ for all $i\in\left\{ 1,\ldots,n-1\right\} $
and for all natural numbers $n$, assemble to define a functor $\mathfrak{P}:\mathfrak{U}\boldsymbol{\beta}\rightarrow FI$.
The functor $\mathfrak{P}$ is strict monoidal with respect to the
monoidal structures $\left(\mathfrak{U}\boldsymbol{\beta},\natural,0\right)$
and $\left(FI,\sqcup,0\right)$. In addition, they define the following
short exact sequence for all natural numbers $n$ (see for example
\cite{birmanbraids}):
\[
\xymatrix{1\ar@{->}[r] & \mathbf{PB}_{n}\ar@{->}[r] & \mathbf{B}_{n}\ar@{->}[r]^{\mathfrak{p}_{n}} & \mathfrak{S}_{n}\ar@{->}[r] & 1}
.
\]
Let $\mathbf{PB}_{-}:\left(\mathbb{N},\leq\right)\rightarrow\mathfrak{Gr}$
and $\mathbf{B}_{-}:\left(\mathbb{N},\leq\right)\rightarrow\mathfrak{Gr}$
be the families of groups defined by $\mathbf{PB}_{-}\left(n\right)=\mathbf{PB}_{n}$,
$\mathbf{B}_{-}\left(n\right)=\mathbf{B}_{n}$ and $\mathbf{B}_{-}\left(\gamma{}_{n}\right)=\mathbf{PB}_{-}\left(\gamma{}_{n}\right)=id_{1}\natural-$
for all natural numbers $n$. Therefore, by Theorem \ref{cor:Mainresultparticularkernels}:
\begin{prop}
\label{prop:Bigresultbraidgroup}Let $F$ be an object of $\mathbf{Fct}\left(\mathfrak{U}\boldsymbol{\beta},\mathbb{K}\textrm{-}\mathfrak{Mod}\right)$.
For all natural numbers $q$, $H_{q}\left(\mathbf{B}_{\infty},F_{\infty}\right)\cong\underset{n\in FI}{Colim}\left(H_{q}\left(\mathbf{PB}_{n},F\left(n\right)\right)\right)$,
and if $F$ factors through the category $FI$, then:
\[
H_{q}\left(\mathbf{B}_{\infty},F_{\infty}\right)\cong\underset{n\in FI}{Colim}\left(H_{q}\left(\mathbf{PB}_{n},\mathbb{K}\right)\underset{\mathbb{K}}{\otimes}F\left(n\right)\right).
\]
\end{prop}

The rational cohomology ring of the pure braid group on $n\in\mathbb{N}$
strands is computed by Arnol'd in \cite{arnold}. Namely, $H^{q}\left(\mathbf{PB}_{n},\mathbb{Q}\right)$
is the graded exterior algebra generated by the classes $\omega_{i,j}$
for $i,j\in\left\{ 1,\ldots,n\right\} $ and $i<j$, subject to the
relations $\omega_{i,j}\omega_{j,k}+\omega_{j,k}\omega_{k,i}+\omega_{k,i}\omega_{i,j}=0$.
By the universal coefficient theorem for cohomology and as $H^{q}\left(\mathbf{PB}_{n},\mathbb{K}\right)$
is a finite-dimensional vector space, we deduce that $H_{q}\left(\mathbf{PB}_{n},\mathbb{K}\right)\cong H^{q}\left(\mathbf{PB}_{n},\mathbb{K}\right)$.

Moreover, the $FI$-module structure of the homology groups $H_{q}\left(\mathbf{PB}_{-},\mathbb{K}\right)$
is well-known by \cite[Example 5.1.A]{ChurchEllenbergFarbstabilityFI}:
the conjugation action of the symmetric group $\mathfrak{S}_{n}$
on $H_{q}\left(\mathbf{PB}_{n},\mathbb{K}\right)$ translates into
the permutation action of $\mathfrak{S}_{n}$ on the indices $i,j\in\left\{ 1,\ldots,n\right\} $
of the generators $\left\{ \omega_{i,j}\right\} _{i,j\in\left\{ 1,\ldots,n\right\} }$.
Hence, fixing some $F\in\mathbf{Fct}\left(FI,\mathbb{K}\textrm{-}\mathfrak{Mod}\right)$,
we have a complete description of the $FI$-module structure of the
functor $H_{q}\left(\mathbf{PB}_{-},\mathbb{K}\right)\underset{\mathbb{K}}{\otimes}F\left(-\right)$.

\subsubsection{Mapping class group of orientable surfaces\label{subsec:Mapping-class-groupappl}}

We take the notations of Section \ref{subsec:FirstsectionMapping-class-groups}.
Recall that we introduced the groupoid $\mathfrak{M}_{2}$ associated
with the surfaces $\varSigma_{n,1}^{s}$ for all natural numbers $n$
and $s$. Let $\mathfrak{M}_{2}^{=}$ be the full subgroupoid of $\mathfrak{M}_{2}$
on the objects $\left\{ \varSigma_{n,1}^{n}\right\} _{n\in\mathbb{N}}$.
By \cite[Proposition 5.18]{WahlRandal-Williams}, the boundary connected
sum $\natural$ induces a strict braided monoidal structure $\left(\mathfrak{M}_{2}^{=},\natural,\left(\varSigma_{0,1}^{0},I\right)\right)$
such that $\mathfrak{U}\mathfrak{M}_{2}^{=}$ is a pre-braided homogeneous
category\textit{.}

Recall that $\textrm{Diff}^{\textrm{\ensuremath{\partial_{0}},points}}\left(\varSigma_{n,1}^{n}\right)$
denotes the space of diffeomorphisms of the surface $\varSigma_{n,1}^{n}$
which fix the boundary pointwise and fix the marked points. Let $\textrm{Diff}^{\textrm{\ensuremath{\partial_{0}},permute }}\left(\varSigma_{n,1}^{n}\right)$
be the space of diffeomorphisms of the surface $\varSigma_{n,1}^{n}$
which fix the boundary pointwise and permute the marked points. The
injections $\left\{ \Gamma_{n,1}^{\left[n\right]}\overset{\mathfrak{i}_{n}}{\hookrightarrow}\Gamma_{n,1}^{n}\right\} _{n\in\mathbb{N}}$
induced by the inclusions
\[
\textrm{Diff}^{\textrm{\ensuremath{\partial_{0}},points}}\left(\varSigma_{n,1}^{n}\right)\hookrightarrow\textrm{Diff}^{\textrm{\ensuremath{\partial_{0}},permute }}\left(\varSigma_{n,1}^{n}\right)
\]
provide the following short exact sequence for all natural numbers
$n$:
\[
\xymatrix{1\ar@{->}[r] & \Gamma_{n,1}^{\left[n\right]}\ar@{->}[r]^{\mathfrak{i}_{n}} & \Gamma_{n,1}^{n}\ar@{->}[r]^{\mathfrak{pm}_{n}} & \mathfrak{S}_{n}\ar@{->}[r] & 1}
.
\]
The surjections $\left\{ \mathfrak{pm}_{n}\right\} _{n\in\mathbb{N}}$
define a strict monoidal functor $\mathfrak{M}_{2}^{=}\rightarrow\Sigma$.
Let $\Gamma_{-,1}^{\left[-\right]}:\left(\mathbb{N},\leq\right)\rightarrow\mathfrak{Gr}$
and $\Gamma_{-,1}^{-}:\left(\mathbb{N},\leq\right)\rightarrow\mathfrak{Gr}$
be the families of groups defined by $\Gamma_{-,1}^{\left[-\right]}\left(n\right)=\Gamma_{n,1}^{\left[n\right]}$,
$\Gamma_{-,1}^{-}\left(n\right)=\Gamma_{n,1}^{n}$ and $\Gamma_{-,1}^{\left[-\right]}\left(\gamma{}_{n}\right)=\Gamma_{-,1}^{-}\left(\gamma{}_{n}\right)=id_{1}\natural-$
for all natural numbers $n$.

From Theorem \ref{thm:resultBodiTil}, we deduce that for all natural
numbers $q$ such that $n\geq2q$:
\[
H_{q}\left(\Gamma_{n,1}^{\left[n\right]},\mathbb{K}\right)\cong\underset{k+l=q}{\bigoplus}\left(H_{k}\left(\Gamma_{n,1},\mathbb{K}\right)\underset{\mathbb{K}}{\otimes}H_{l}\left(\left(\mathbb{C}P^{\infty}\right)^{\times n},\mathbb{K}\right)\right).
\]
The conjugation action of the symmetric group $\mathfrak{S}_{n}$
on $\Gamma_{n,1}^{\left[n\right]}$ is induced by the natural action
of $\mathfrak{S}_{n}$ on $\varSigma_{n,1}^{n}$ given by permuting
the marked points. Hence, according to the decomposition of the classifying
space associated with the pure mapping class groups in \cite[Theorem 1]{boedigtilman2},
the action of $\mathfrak{S}_{n}$ on $H_{q}\left(\Gamma_{n,1}^{\left[n\right]},\mathbb{K}\right)$
corresponds to permuting the $n$ factors $\mathbb{C}P^{\infty}$:
the $FI$-module structure of the homology groups $H_{q}\left(\Gamma_{n,1}^{\left[n\right]},\mathbb{K}\right)$
is thus well-understood using Künneth formula for $H_{l}\left(\left(\mathbb{C}P^{\infty}\right)^{\times n},\mathbb{Z}\right)$.
A fortiori the homology group $H_{*}\left(\Gamma_{n,1},\mathbb{K}\right)$
is a trivial $\mathfrak{S}_{n}$-module. Recall also from \cite{MadsenWeiss}
that:
\[
H_{*}\left(\Gamma_{\infty,1},\mathbb{K}\right)\cong\mathbb{K}\left[\kappa_{1},\kappa_{2},\ldots\right]
\]
where each $\kappa_{i}$ has degree $2i$.

By Theorem \ref{cor:Mainresultparticularkernels}, we deduce that:
\begin{prop}
\label{prop:Bigresultmcgorient}Let $F$ be an object of $\mathbf{Fct}\left(\mathfrak{U}\mathfrak{M}_{2}^{=},\mathbb{K}\textrm{-}\mathfrak{Mod}\right)$.
For all natural numbers $q$,
\[
H_{q}\left(\Gamma_{\infty,1}^{\infty},F_{\infty}\right)\cong\underset{n\in FI}{Colim}\left(H_{q}\left(\Gamma_{n,1}^{\left[n\right]},F\left(n\right)\right)\right).
\]
In particular, if $F$ factors through the category $FI$, then:
\[
H_{q}\left(\Gamma_{\infty,1}^{\infty},F_{\infty}\right)\cong\underset{n\in FI}{Colim}\left(\underset{k+l=q}{\bigoplus}\left(H_{k}\left(\Gamma_{n,1},\mathbb{K}\right)\underset{\mathbb{K}}{\otimes}H_{l}\left(\left(\mathbb{C}P^{\infty}\right)^{\times n},\mathbb{K}\right)\right)\underset{\mathbb{K}}{\otimes}F\left(n\right)\right),
\]
and a fortiori $H_{2k+1}\left(\Gamma_{\infty,1}^{\infty},F_{\infty}\right)=0$
for all natural numbers $k$.
\end{prop}

\begin{rem}
The key point for the calculations of Section \ref{subsec:Mapping-class-groupappl}
to work is to consider a subgroupoid $\mathfrak{M}_{2}^{'}$ of $\mathfrak{M}_{2}$
so that the braided monoidal structure $\natural$ restricts to it
(i.e $\left(\mathfrak{M}_{2}^{'},\natural,\left(\varSigma_{0,1}^{0},I\right)\right)$
is a braided monoidal groupoid) and so that the number of marked evolves
linearly with respect to the genus: for instance we could consider
the category with objects $\left\{ \varSigma_{n,1}^{2n}\right\} _{n\in\mathbb{N}}$
and all the above results follow mutatis mutandis. However this does
not change the result for the colimit.
\end{rem}

\subsubsection{Particular right-angled Artin groups}

A right-angled Artin group (abbreviated RAAG) is a group with a finite
set of generators $\left\{ s_{i}\right\} _{1\leq i\leq k}$ with $k\in\mathbb{N}$
and relations $s_{i}s_{j}=s_{j}s_{i}$ for some $i,j\in\left\{ 1,\ldots,n\right\} $.
For instance, the free group on $k$ generators $\mathbf{F}_{k}$
is a RAAG. We refer to \cite{VogtmannRaags} or \cite[Section 3]{gandiniwahl}
for more details on these groups.

By \cite[Proposition 3.1]{gandiniwahl}, any RAAG admits a maximal
decomposition as a direct product of RAAGs, unique up to isomorphism
and permutation of the factors. A RAAG is said to be unfactorizable
if its maximal decomposition is itself. We refer to \cite{VogtmannRaags}
or \cite[Section 3]{gandiniwahl} for more details on these groups.
Moreover, we have the following key property:
\begin{lem}
\cite[Proposition 3.3]{gandiniwahl} Let $A$ be a fixed unfactorizable
RAAG different from $\mathbb{Z}$. For all natural numbers $n$, we
have the following split short exact sequence:
\[
\xymatrix{1\ar@{->}[r] & Aut\left(A\right)^{\times n}\ar@{->}[r]^{i_{n}} & Aut\left(A^{\times n}\right)\ar@{->}[r]^{\textrm{{\color{white}iiiiiii}}s_{n}} & \mathfrak{S}_{n}\ar@{->}[r] & 1.}
\]
\end{lem}

Let $\mathcal{R}_{A}$ be the groupoid with the groups $A^{\times n}$
for all natural numbers $n$ as its objects and $Aut\left(A^{\times n}\right)$
as automorphism groups. By \cite[Sections 1 and 5]{gandiniwahl},
the direct product $\times$ induces a strict symmetric monoidal structure
$\left(\mathcal{R}_{A},\times,0_{\mathfrak{Gr}}\right)$ such that
$\mathfrak{U}\mathcal{R}_{A}$ is a pre-braided homogeneous category.
It is clear that the surjections $\left\{ s_{n}\right\} _{n\in\mathbb{N}}$
define a strict monoidal functor $S:\mathcal{R}_{A}\rightarrow\Sigma$.
Let $Aut\left(A^{\times-}\right):\left(\mathbb{N},\leq\right)\rightarrow\mathfrak{Gr}$
and $Aut\left(A\right)^{\times-}:\left(\mathbb{N},\leq\right)\rightarrow\mathfrak{Gr}$
be the families of groups defined by $Aut\left(A^{\times-}\right)\left(n\right)=Aut\left(A^{\times n}\right)$,
$Aut\left(A\right)^{\times-}\left(n\right)=Aut\left(A\right)^{\times n}$
and $Aut\left(A^{\times-}\right)\left(\gamma{}_{n}\right)=Aut\left(A\right)^{\times-}\left(\gamma{}_{n}\right)=id_{1}\times-$
for all natural numbers $n$. By Theorem \ref{cor:Mainresultparticularkernels}:
\begin{prop}
\label{prop:BigresultRAAG}Let $F$ be an object of $\mathbf{Fct}\left(\mathfrak{U}\mathcal{R}_{A},\mathbb{K}\textrm{-}\mathfrak{Mod}\right)$
and $A$ be a fixed unfactorizable right-angled Artin group different
from $\mathbb{Z}$. For all natural numbers $q$, $H_{q}\left(Aut\left(A^{\times\infty}\right),F_{\infty}\right)\cong\underset{n\in FI}{Colim}\left(H_{q}\left(Aut\left(A\right)^{\times n},F\left(n\right)\right)\right)$,
and if $F$ factors through the category $FI$, then:
\begin{equation}
H_{q}\left(Aut\left(A^{\times\infty}\right),F_{\infty}\right)\cong\underset{n\in FI}{Colim}\left(H_{q}\left(Aut\left(A\right)^{\times n},\mathbb{K}\right)\underset{\mathbb{K}}{\otimes}F\left(n\right)\right).\label{eq:homologyRAAG}
\end{equation}
\end{prop}

\begin{cor}
\label{cor:BigresultRAAG}Let $A$ be a fixed unfactorizable right-angled
Artin group different from $\mathbb{Z}$, such that there exists $N_{A}\in\mathbb{N}$
such that $H_{q}\left(Aut\left(A\right),\mathbb{K}\right)=0$ for
$1\leq q\leq N_{A}$. Then, for all objects $F$ of $\mathbf{Fct}\left(\mathfrak{U}\mathcal{R}_{A},\mathbb{K}\textrm{-}\mathfrak{Mod}\right)$
factoring through the category $FI$:
\[
H_{q}\left(Aut\left(A^{\times\infty}\right),F_{\infty}\right)=0,
\]
for all natural numbers $q$ such that $1\leq q\leq N_{A}$.
\end{cor}

\begin{proof}
It follows from Künneth Theorem that, for all natural numbers $q$
such that $1\leq q\leq N_{A}$, $H_{q}\left(Aut\left(A\right)^{\times n},\mathbb{K}\right)=0$.
Then, the result follows from (\ref{eq:homologyRAAG}).
\end{proof}
\begin{example}
\label{exa:caseofautfree}Let $\mathbf{F}_{k}$ be the free group
on $k$ generators. By \cite[Corollary 1.2]{Galatius}, for $k\geq2q+1$
and $q\neq0$, $H_{q}\left(Aut\left(\mathbf{F}_{k}\right),\mathbb{K}\right)=0$.
Let $F$ be an object of $\mathbf{Fct}\left(\mathfrak{U}\mathcal{R}_{\mathbf{F}_{k}},\mathbb{K}\textrm{-}\mathfrak{Mod}\right)$
factoring through the category $FI$. Then, for all natural numbers
$q$ and $k$ such that $1\leq q\leq\frac{k-1}{2}$:
\[
H_{q}\left(Aut\left(\left(\mathbf{F}_{k}\right)^{\times\infty}\right),F_{\infty}\right)=0.
\]
In particular, $H_{q}\left(Aut\left(\left(\mathbf{F}_{\infty}\right)^{\times\infty}\right),F_{\infty}\right)=0$
for all $FI$-module $F$.
\end{example}

\bibliographystyle{plain}
\bibliography{bibliographiethese}

\lyxaddress{\textit{E-mail address: }\texttt{artsou@hotmail.fr}}
\end{document}